\documentclass[oneside,english]{amsart}
\usepackage[T1]{fontenc}
\usepackage[latin9]{inputenc}
\usepackage{color}
\usepackage{amstext}
\usepackage{amsthm}
\usepackage{amssymb}

\makeatletter
\numberwithin{equation}{section}
\numberwithin{figure}{section}

\usepackage{bbm}
\usepackage{tikz-cd}
\usepackage{xypic}
\usepackage{hyperref}
\usepackage{color}
\usepackage{amstext}
\usepackage{mathtools}
\usepackage{amsmath, amscd, amsthm, amssymb, graphics, xypic, mathrsfs, setspace, fancyhdr, bm, pdfsync, mathptmx}
\usepackage{enumitem}
\usepackage[a4paper,top=1.15in, bottom=1.15in, left=0.95in, right=0.95in]{geometry}

\newcommand\C{{\mathbb C}}

\newcommand\PGL{{\mathrm{PGL}}}

\newcommand\R{{\mathbb R}}
\newcommand\A{{\mathbb A}}
\newcommand\Z{{\mathbb Z}}

\newcommand{\Aut}{\mathrm{Aut}}

\theoremstyle{plain}
\newtheorem{thm}{\protect\theoremname}
\theoremstyle{definition}

\theoremstyle{definition}
\theoremstyle{plain}

\theoremstyle{remark}
\newtheorem{rem}[thm]{\protect\remarkname}
\theoremstyle{plain}

\theoremstyle{plain}
\newtheorem{theorem}{Theorem}[section]
\newtheorem{theorem*}{Theorem}
\newtheorem{proposition}[theorem]{Proposition}
\newtheorem{lemma}[theorem]{Lemma}
\newtheorem{corollary}[theorem]{Corollary}

\newtheorem*{question*}{Question}
\theoremstyle{definition}
\newtheorem{remark}[theorem]{Remark}
\newtheorem{definition}[theorem]{Definition}

\newtheorem{notation}[theorem]{Notation}
\newtheorem{example}[theorem]{Example}
\makeatother

\usepackage{babel}

\providecommand{\definitionname}{Definition}

\providecommand{\lemmaname}{Lemma}
\providecommand{\propositionname}{Proposition}

\providecommand{\remarkname}{Remark}
\providecommand{\theoremname}{Theorem}

\usepackage[normalem]{ulem}

\begin{document}
\title[]{Relative forms of real algebraic varieties and examples of quasi-projective surfaces with algebraic moduli of real forms}

\author{Anna Bot} \address{University of Basel, Department of Mathematics and Computer Science, Spiegelgasse 1, 4057 Basel, Switzerland} \email{annakatharina.bot@unibas.ch}
\author{Adrien Dubouloz} \address{IMB UMR 5584, CNRS, Univ. Bourgogne, F-21000 Dijon, France.} \email{adrien.dubouloz@u-bourgogne.fr}
\thanks{The first author acknowledges support by the Swiss National Science Foundation Grant ``Geometrically ruled surfaces'' 200020-1 92217. The second author was partially supported by the French ANR project ``FIBALGA'' ANR-18-CE40-0003. The Institute of Mathematics of Burgundy receives support from the EIPHI Graduate School ANR-17-EURE-0002.}


\maketitle
\begin{abstract} We propose a framework to give a precise meaning to the intuitive notion of ``family of real forms of a variety parametrised by a variety'' and study some fundamental properties of this notion. As an illustration, for any $n\geq 1$, we construct the first example of a quasi-projective real surface whose mutually non-isomorphic real forms admit a moduli of dimension at least $n$, parametrised by the real points of an affine $n$-space. Expanding on these constructions, we can give quasi-projective real varieties of any dimension whose algebraic moduli of the non-isomorphic real forms has arbitrarily positive dimension.
\end{abstract} 

\section*{Introduction}

\vspace{0.5cm}

After a pioneering breakthrough due to Lesieutre~\cite{Le18} who constructed the very first example of smooth real projective varieties with countably infinitely many real forms, considerable progress has been made recently towards the understanding of which real algebraic varieties have finitely or infinitely many real forms up to isomorphism. Of particular interest is the case of rational surfaces, especially projective ones, motivated by a quite long standing question raised by Kharlamov~\cite{Kha02} about the existence of such surfaces with infinitely many real forms. Examples answering the latter question in the positive have been announced by Dinh-Oguiso-Yu~\cite{dinh2021smooth}. The situation for rational affine surfaces has been considered by the first author in~\cite{bot2021smooth}, in which examples of smooth rational real affine surfaces which even admit uncountably many pairwise non-isomorphic real forms are constructed. 
The fact that these real forms appear to be indexed by the points of a finite-to-one quotient of an open subset of the real line $\mathbb{R}$ motivates the question of a moduli interpretation of real forms of a given quasi-projective real algebraic variety $X$. Ideally, such an interpretation should take the form of a natural refinement of the classical set-theoretic correspondence between isomorphism classes of real forms of $X$ and elements of the Galois cohomology set $H^1(\mathrm{Gal}(\C/\R),\mathrm{Aut}_{\C}(X_\C))$ of cohomology classes of Galois $1$-cocycles with values in the automorphism group $\mathrm{Aut}_\C(X_\C)$ of the complexification $X_\C$ of $X$, which captures additional algebraic and topological properties.

\medskip

Our goal in this paper is twofold: first to propose a framework to give a precise meaning to the intuitive notion of ``family of real forms of a variety parametrised by a variety'' and to study some fundamental properties of this notion. Second, to produce in this framework new examples of varieties with uncountably many pairwise non-isomorphic real forms whose respective moduli can have arbitrarily high dimension.
 It is known (see~\cite{DinetAl22,La22} and Theorem~\ref{thm:Quasiproj-repAut}) that such varieties must be sought among non-projective varieties of dimension larger than or equal to two and with non-representable automorphism groups. We apply our methods to show that for any $n\geq1$, there exists a smooth quasi-projective real toric surface having uncountably many families of pairwise non-isomorphic real forms parametrised by the real points of $\A^n_{\R}$. These real forms have the additional property to be topologically indistinguishable, their real loci being indeed all diffeomorphic to the Euclidean plane $\R^2$.  
This result can be extended to show that for any pair $(d,m)$ with $d\geq2$ and $m\geq 1$, there exists a smooth quasi-projective real algebraic variety which admits uncountably many families of pairwise non-isomorphic real forms all birationally diffeomorphic (see Subsection~\ref{subsec:diffbir} for the definition) to $\A^d_{\R}$ parametrised by the real points of $\A^m_{\R}$.

\medskip

Let us now describe more precisely the content and the main results of the article. 

A step towards a moduli viewpoint on real forms is to consider the notion of relative real form of a given real algebraic variety $X$ over a positive dimensional real algebraic variety $S$: roughly, one can define an \emph{$S$-form of $X$} as being  an $S$-scheme $f\colon\mathcal{X}\to S$ which is a real form  of $S\times_{\mathrm{Spec}(\R)}X$ over $S$, in the sense that the complexification $f_\C\colon\mathcal{X}_\C\to S_\C$ of $\mathcal{X}$ is isomorphic to $S_\C\times_{\mathrm{Spec}(\C)} X_\C$ as a scheme over the complexification $S_\C$ of $S$ (see Subsection~\ref{subsec:familie-basic}  for a more precise definition of this notion). The fibres of an $S$-form $f\colon\mathcal{X}\to S$ of $X$ over $\R$-rational points of $S$ being real forms of $X$ in the usual sense, one can informally interpret $f\colon\mathcal{X}\to S$ as a collection of real forms of $X$ varying algebraically over the $\R$-rational points of $S$. From the Galois descent viewpoint, an $S$-form of $X$ is equivalently determined under suitable quasi-projectivity assumptions by a choice of lift of the canonical descent datum on $S_\C$ with respect to the Galois extension $\mathrm{Spec}(\C)\to \mathrm{Spec}(\R)$ to Galois descent data on $S_\C\times_{\mathrm{Spec}(\C)} X_\C$. Such a choice of lift can in turn be thought of as describing a collection of Galois $1$-cocycles with values in the automorphism group of $X_\C$ varying algebraically over the $\R$-rational points  of $S$. A natural question in this context then concerns the existence of a ``universal relative form'' of a given real algebraic variety. Putting aside representability problems related to the question of effectiveness of given Galois descent data and building on the fact that the notions of real forms and relative real forms of varieties carry almost verbatim over to arbitrary set-valued contravariant functors on the category of $\R$-schemes, we are able to answer this existence question in a way which can be summarised as follows (see Subsection~\ref{subsec:Tauto-form} for precise statements):

\begin{theorem*}\label{MainThm-0} For every $\R$-scheme $X$ there exists a morphism $u\colon\mathfrak{U}_X\to \mathfrak{Z}$ of set-valued contravariant functors on the category of $\R$-schemes with the following properties:
 \begin{enumerate}[label=(\alph*), leftmargin=*] 
 	\item The functor $\mathfrak{Z}$ is a real form of the automorphism group functor $\underline{\mathrm{Aut}}_\R(X)$ of $X$ depending only on the group functor structure on $\underline{\mathrm{Aut}}_\R(X)$ and whose set $\mathfrak{Z}(\mathrm{Spec}(\R))$ of $\R$-points is in natural bijection with the set $Z^1(\mathrm{Gal}(\C/\R),\mathrm{Aut}_{\C}(X_\C))$ of Galois $1$-cocycles with values in the group $\mathrm{Aut}_{\C}(X_\C)$.
    \item The morphism $u\colon\mathfrak{U}_X\to \mathfrak{Z}$ is a complete family of relative real forms of $X$ in the sense that every relative form $f\colon\mathcal{X}\to S$ of $X$ over an $\R$-scheme $S$ is obtained from $u\colon\mathfrak{U}_X\to \mathfrak{Z}$ as a suitable base change by a morphism $\Psi\colon S\to \mathfrak{Z}$. 
    \end{enumerate}
\end{theorem*}  

A consequence of Theorem~\ref{MainThm-0} is that every real form $X'$ of $X$ appears as the fibre of $u\colon\mathfrak{U}_X\to \mathfrak{Z}$ over the $\R$-point of $\mathfrak{Z}$ corresponding to the Galois $1$-cocycle with values in $\mathrm{Aut}_\C(X_\C)$ which determines $X'$. To complete the correspondence with the classification of real forms in terms of Galois cohomology, we show in addition that the functor $\mathfrak{Z}$ can be equipped with an action of the ``Weil restriction'' to the category of group-valued contravariant functors on the category of $\R$-schemes of the automorphism group functor $\underline{\mathrm{Aut}}_\C(X_\C)$ of the complexification of $X$, which again depend only on the group functor structure of $\underline{\mathrm{Aut}}_\R(X)$, and which has the property that the orbits of the $\R$-points of $\mathfrak{Z}$ are in natural bijection with the elements of the Galois cohomology set  $H^1(\mathrm{Gal}(\C/\R),\mathrm{Aut}_{\C}(X_\C))$. 

\medskip

For an arbitrary $\R$-scheme $X$, the functors $\mathfrak{Z}$ and $\mathfrak{U}_X$  are in general not representable, a necessary condition for the representability of $\mathfrak{Z}$ being in particular the representability of the automorphism group functor $\underline{\mathrm{Aut}}_\C(X_\C)$. Nevertheless, we establish the following result which, when combined with~\cite{MO67}, implies in particular that $\mathfrak{Z}$ and $\mathfrak{U}_X$ are representable for every projective $\R$-scheme $X$ (see Theorem~\ref{thm:Quasiproj-repAut} for a more complete statement): 
~
\begin{theorem*}\label{MainThm-0-1}Let $X$ be a quasi-projective $\R$-scheme whose automorphism group functor $\underline{\mathrm{Aut}}_\R(X)$  is representable by a group $\R$-scheme  locally of finite type with at most countably many irreducible components. Then the functors  $\mathfrak{Z}$ and $\mathfrak{U}_X$  are representable by $\R$-schemes locally of finite type with quasi-projective connected components.
\end{theorem*}

A reinterpretation of two recent results~\cite[{Theorem~1.2}]{DinetAl22},~\cite[{Theorem~1.1}]{La22} implies in addition that in the setting of Theorem~\ref{MainThm-0-1}, the quasi-projective $\R$-scheme $X$ has at most countably many non-isomorphic real forms and in fact even finitely many whenever $\underline{\mathrm{Aut}}_\R(X)$ is representable by an algebraic group $\R$-scheme. {In particular, a projective $\R$-scheme $X$ has at most countably many isomorphism classes of real forms. In the case where $\underline{\mathrm{Aut}}_\R(X)$ is representable by an algebraic group $\R$-scheme $G$ \footnote{Say for instance when $X$ is projective with ample or anti-ample canonical sheaf, in which case $G$ is a linear algebraic group $\R$-scheme.} then $X$ has finitely many isomorphism classes of real forms; our construction provides an interesting refined interpretation of what is merely captured by the finiteness of the Galois cohomology set. Namely, in this situation, the real form $Z$ of $G$ representing the functor $\mathfrak{Z}$ is a smooth real quasi-projective variety whose real locus $Z(\R)$ endowed with its Euclidean topology is a smooth manifold whose points are in natural bijective correspondence with elements of the set $Z^1(\mathrm{Gal}(\C/\R),\mathrm{Aut}_{\C}(X_\C))$ of Galois $1$-cocycles. Furthermore, it comes equipped with an algebraic action of the Weil restriction $R_{\C/\R}G_\C$ of the complexification $G_\C$ of $G$ --- which is again a quasi-projective group $\R$-scheme --- with the property that two Galois $1$-cocycles have the same class in  $H^1(\mathrm{Gal}(\C/\R),\mathrm{Aut}_{\C}(X_\C))$ if and only if the corresponding points of $Z(\R)$ belong to the same orbit of the induced differential action of the real Lie group $R_{\C/\R}G_\C(\R)\cong G_\C(\C)$ endowed with its Euclidean topology.

\medskip

As an illustration and application of the notion of relative real forms we consider in the second part of the article the rational quasi-projective surfaces $X_n$, $n\geq 0$, in $\mathbb{A}^1_\R \times \mathbb{P}^2_\R=\mathrm{Proj}_{\mathbb{R}[t]}(\mathbb{R}[t][u,v,w])$ defined by the equations \[u^2+v^2-t^{2n+1}w^2=0.\] The surface $X_n$ is smooth if $n=0$ and singular otherwise, containing a unique rational double point $(0,[0:0:1])$ whose complexification is a Du Val singularity of type $A_{2n}$. Furthermore, the map \[\mathbb{R}^2\to X_n(\mathbb{R}),\quad (x_1,x_2)\mapsto (\sqrt[2n+1]{x_1^2+x_2^2},[x_1:x_2:1])\] is a homeomorphism onto the set of $\mathbb{R}$-rational points of $X_n$ endowed with the Euclidean topology inherited from that of $\mathbb{R}\times \mathbb{RP}^2$. In particular, $X_n$ is a rational quasi-projective model of $\mathbb{R}^2$ in the sense of~\cite{DuMa17} and it can verified that the same holds for the minimal desingularisation $\tilde{X}_n \to X_n$ of $X_n$. Our interest for these surfaces is motivated by the fact that the identity components of the automorphisms groups $\mathrm{Aut}_{\C}(\tilde{X}_{n,\C})$ of their complexifications contain infinite dimensional, hence non-algebraic, subgroups of $\mathrm{PGL}_2(\C[t^{\pm  1}])$ from which we are able to build certain explicit relative real forms  $f_n\colon\tilde{\mathcal{X}}_n\to \mathbb{A}^m_{\R}$ over affine spaces of arbitrarily large positive dimension $m$ and determine isomorphism classes of their real fibres. We obtain in particular the following result (see Proposition~\ref{prop:main-prop} for a stronger and more precise statement):  

\begin{theorem*}\label{MainThm} For every $n\geq 2$ there exists a relative real form $h_n\colon\tilde{\mathcal{X}}_n\to \mathbb{A}_\R^{n-1}$ of the smooth surface $\tilde{X}_n$ with the following properties: 
\begin{enumerate}[label=(\alph*), leftmargin=*] 
	\item The fibres $h_n$ over $\R$-rational points of  $\mathbb{A}_\R^{n-1}$ are pairwise non-isomorphic real forms of $\tilde{X}_n$.  \label{item: MainThm, 2} 
	\item The induced map $h_n(\mathbb{R})\colon\tilde{\mathcal{X}}_n(\mathbb{R})\to \mathbb{A}^{n-1}_\R(\R)\cong \mathbb{R}^{n-1}$ between the sets of  $\R$-rational points of $\tilde{\mathcal{X}}_n$ and $\mathbb{A}^{n-1}_\R$ endowed with their respective Euclidean topologies is diffeomorphic to the trivial bundle $\mathbb{R}^2\times \mathbb{R}^{n-1}$. \label{item: MainThm, 3} 
\end{enumerate} 
\end{theorem*} 

While it is not our purpose to provide a complete classification of all real forms of the real smooth surfaces $\tilde{X}_n$ for all $n\geq 0$, the above result should be informally interpreted as the property that for $n\geq 2$, all surfaces $\tilde{X}_n$ do not only have uncountably pairwise non-isomorphic real forms but also that the number of moduli of these real forms is at least equal to $n-1$; in particular it grows to infinity with the integer $n$.  In fact, one may ask:
\begin{question*} 
What is the full description of the complete moduli of $\tilde{X}_n$? Can one find a fixed quasi-projective real smooth variety --- maybe $\tilde{X}_n$ --- whose moduli of real forms cannot be bounded in dimension?
\end{question*}

Assertion~\ref{item: MainThm, 3} in Theorem~\ref{MainThm} says in particular that in contrast with the countably infinite collection of real forms of the real affine fourfold $\{x^2+y^2+z^2=1\}\times \mathbb{A}^2_\R$ constructed in~\cite{DFMJ21} (see Example~\ref{exa:DFM}) which have pairwise non-homeomorphic real loci, the uncountable collections of real forms of $\tilde{X}_n$ we construct have real loci all diffeomorphic to $\mathbb{R}^2$. As in~\cite{bot2021smooth}, for fixed $n$, the real forms of the smooth surfaces $\tilde{X}_n$ we consider have the stronger property to be all birationally diffeomorphic to each other; as a matter of fact, they are all birationally diffeomorphic to $\mathbb{A}^2_\R$. Taking as in~\cite{DFMJ21} appropriate products with well-chosen smooth real affine varieties of log-general type birationally diffeomorphic to real affine spaces allows in turn to derive the following result (see Proposition~\ref{prop:Main-Cor}):

\begin{theorem*}\label{MainThm-2} For every $d\geq 2$ and every $m\geq 1$ there exists a smooth rational real quasi-projective variety $X$ of dimension $d$ and a relative real form $h\colon\mathcal{X}\to \mathbb{A}_\mathbb{R}^{m}$ of $X$ whose fibres over $\R$-rational points of $\mathbb{A}_\mathbb{R}^{m}$ are pairwise non-isomorphic real forms of $X$ all birationally diffeomorphic to $\mathbb{A}_\R^d$.
\end{theorem*} 

\medskip

The article is organised as follows. In the first two sections, we collect basic properties on functors with Galois descent data on complex schemes, introduce the general notion of relative real forms of an $\R$-scheme $X$ and construct the associated complete relative real form $u_X\colon\mathfrak{U}_X\to \mathfrak{Z}$ of Theorem~\ref{MainThm-0}. Most of the notions, constructions and results in these two sections actually build on general properties of fpqc descent applied to Galois extensions of fields and hence can be carried on almost verbatim for arbitrary, not necessarily finite, Galois extensions rather than just the quadratic extension $\R \subset \C$. Nevertheless, we have made the choice to stick to the case of real and complex schemes for the reason that in this context, the general machinery of fpqc descent data on   complex schemes and contravariant functors on them\footnote{Equivalently, of continuous Galois descent data in the sense of~\cite{Se}.} can be skimmed to take the convenient form of a single datum called a \emph{real structure} and an associated formalism which has been traditionally developed and used in real algebraic geometry over the past decades.  The third section is devoted to the study of geometric properties of the surfaces $X_n$, their desingularisations and their respective complexifications. We give in particular a complete description of their automorphism groups which we use to construct relative forms of these surfaces. We then study isomorphism types and homeomorphism types of real fibres of these relative forms and proceed to the proofs of Theorem~\ref{MainThm} and Theorem~\ref{MainThm-2}.  

\section{Preliminaries}

For $k=\mathbb{R},\mathbb{C}$, we denote by $\mathcal{F}_{k}$ the
category of set-valued contravariant functors on the category $(\mathrm{Sch}_{/k})$
of $k$-schemes. For a $k$-scheme $X$, we let $\underline{X}=\mathrm{Hom}_{k}(-,X)\in\mathrm{Ob}(\mathcal{F}_{k})$
be the functor of points of $X$. For a pair consisting of $k$-scheme
$X$ and a functor $F\in\mathrm{Ob}(\mathcal{F}_{k})$ we identify
the sets $F(X)$ and $\mathrm{Hom}_{\mathcal{F}_{k}}(\underline{X},F)$.
For a pair of $k$-schemes $X$ and $X'$ , we use the same letter
to refer to a morphism $f:\underline{X'}\to\underline{X}$ in $\mathcal{F}_{k}$
and to the morphism $f(X')(\mathrm{id}_{X'}):X'\to X$ in $(\mathrm{Sch}_{/k})$.
We use the notations $F\times_{\underline{k}}F'$ and $X\times_{k}X'$
as short hands for the fibre products $F\times_{\underline{\mathrm{Spec}(k)}}F'$
and $X\times_{\mathrm{Spec}(k)}X'$ in the categories $\mathcal{F}_{k}$
and $(\mathrm{Sch}_{/k})$ respectively. The \emph{automorphism group
functor} of a $k$-scheme $X$ is the group-valued contravariant functor
$\underline{\mathrm{Aut}}_{k}(X)$ on the category of $k$-schemes
defined by $\underline{\mathrm{Aut}}_{k}(X)(T)=\mathrm{Aut}_{T}(T\times_{\mathrm{Spec}(k)}X)$
for every $k$-scheme $T$. An algebraic $k$-variety is a reduced
$k$-scheme of finite type. 

\subsection{Conventions and notations}

We consider $\mathrm{Spec}(\mathbb{C})$ as an $\mathbb{R}$-scheme
via the \'etale cover $\nu\colon\mathrm{Spec}(\mathbb{C})\to\mathrm{Spec}(\mathbb{R})$
induced by the inclusion $\mathbb{R}\to\mathbb{C}\coloneqq\mathbb{R}[X]/(X^{2}+1)$
of $\mathbb{R}$ as the subring of constants. Denoting the residue
class of $X$ by $i$, the morphism $\nu$ is a torsor under the constant
finite $\mathbb{R}$-group scheme $\Gamma_{\mathbb{R}}\coloneqq\mathbb{Z}_{2,\mathbb{R}}=\mathrm{Spec}(\mathbb{R}\times\mathbb{R})$
for the action $c\colon\Gamma\times_{\mathrm{Spec}(\mathbb{R})}\mathrm{Spec}(\mathbb{C})\to\mathrm{Spec}(\mathbb{C})$
with comorphism 
\[
\mathbb{R}[X]/(X^{2}+1)\to\mathbb{R}[X]/(X^{2}+1)\otimes_{\mathbb{R}}(\mathbb{R}\times\mathbb{R})=\mathbb{R}[X]/(X^{2}+1)\times\mathbb{R}[X]/(X^{2}+1),\;i\mapsto(i,-i).
\]
We denote by $\tau$ the involution $c(-1,-)$ of $\mathrm{Spec}(\mathbb{C})$
corresponding to the complex conjugation $z\mapsto\overline{z}$ on $\mathbb{C}$.

For an $\mathbb{R}$-scheme $h\colon S\to\mathrm{Spec}(\mathbb{R})$, we
denote by $S_{\mathbb{C}}=\nu^{*}S$ the $\mathbb{C}$-scheme $\mathrm{pr}_{2}\colon S\times_{h,\mathrm{Spec}(\mathbb{R}),\nu}\mathrm{Spec}(\mathbb{C})\to\mathrm{Spec}(\mathbb{C})$.
For a $\mathbb{C}$-scheme $f\colon T\to\mathrm{Spec}(\mathbb{C})$, we
denote by $\nu_{*}T$ the $\mathbb{R}$-scheme $\nu\circ f\colon T\to\mathrm{Spec}(\mathbb{R})$
and by $\overline{T}=\tau_{*}T$ the $\mathbb{C}$-scheme $\tau\circ f\colon T\to\mathrm{Spec}(\mathbb{C})$.
A morphism of $\mathbb{C}$-schemes $\psi\colon T'\to T$ induces a morphism
of $\mathbb{C}$-schemes $\tau_{*}\psi\colon\overline{T'}\to\overline{T}$.
Since $\tau^{2}=\mathrm{id}_{\mathbb{C}}$ and $\nu\circ\tau=\nu$,
for every $\mathbb{C}$-scheme $T$, we have canonical equalities
$\overline{\overline{T}}=T$ and $\nu_{*}\overline{T}=\nu_{*}T$ of $\mathbb{C}$-schemes
and $\mathbb{R}$-schemes respectively. On the other hand, the isomorphism
of $\mathbb{C}$-modules $\mathbb{C}\otimes_{\mathbb{R}}\mathbb{C}\to\mathbb{C}\times\mathbb{C}$,
$a\otimes b\mapsto(ab,a\overline{b})$ induces for every $\mathbb{C}$-scheme
$T$ an isomorphism of $\mathbb{C}$-schemes $(\nu_{*}T)_{\mathbb{C}}=\nu^{*}\nu_{*}T\cong T\sqcup\overline{T}$. 

\subsection{Real functors on complex schemes\label{subsec:real-struct}}

\subsubsection{Standard adjunction }

The morphism $\nu\colon\mathrm{Spec}(\mathbb{C})\to\mathrm{Spec}(\mathbb{R})$
induces a pair of adjoint functors \[
\begin{tikzcd}     \mathcal{F}_{\mathbb{R}} \arrow[bend left=15]{rr}{\nu^*} & \vdash  & \mathcal{F}_{\mathbb{C}}\arrow[bend left=15]{ll}{\nu_*}  
\end{tikzcd} 
\] defined respectively by $(\nu^{*}H)(T)=H(\nu_{*}T)$ for every $\mathbb{C}$-scheme
$T$ and by $(\nu_{*}F)(S)=F(\nu^{*}S)$ for every $\mathbb{R}$-scheme
$S$, the counit--unit adjunction being given by the pair of functors
\begin{align*}
\epsilon & \colon 1_{\mathcal{F_{\mathbb{R}}}}\to\nu_{*}\nu^{*},\,H(S)\stackrel{H(\mathrm{pr}_{1})}{\longrightarrow}(\nu_{*}\nu^{*}H)(S)=H(\nu_{*}S_{\mathbb{C}})\\
\eta & \colon \nu^{*}\nu_{*}\to1_{\mathcal{F}_{\mathbb{C}}},\,(\nu^{*}\nu_{*}F)(T)=F(T\sqcup\overline{T})\stackrel{F(j_{T})}{\longrightarrow}F(T)
\end{align*}
where $\mathrm{pr}_{1}\colon S_{\mathbb{C}}=S\times_{\mathrm{Spec}(\mathbb{R})}\mathrm{Spec}(\mathbb{C})\to\mathrm{Spec}(\mathbb{C})$
is the first projection and $j_{T}\colon T\to T\sqcup\overline{T}$ is the natural
immersion. The functor $\nu^{*}$ commutes with fibre products in
$\mathcal{F}_{\mathbb{R}}$ and $\mathcal{F}_{\mathbb{C}}$ respectively,
namely, if $h'\colon H'\to H$ and $h''\colon H''\to H$ are two morphisms in
$\mathcal{F}_{\mathbb{R}}$ then the canonical morphism $c\colon \nu^{*}(H'\times_{h',H,h''}H'')\to\nu^{*}H'\times_{\nu^{*}h',\nu^{*}H,\nu^{*}h''}\nu^{*}H''$
is an isomorphism. For a functor $H\in\mathrm{Ob}(\mathcal{F}_{\mathbb{R}})$,
we will henceforth often denote $\nu^{*}H$ simply by $H_{\mathbb{C}}$,
similarly for a morphism $h\colon H'\to H$ in $\mathcal{F}_{\mathbb{R}}$
we put $h_{\mathbb{C}}=\nu^{*}h$. 

The morphism $\tau$ induces an endofunctor $\tau^{*}$ of $\mathcal{F}_{\mathbb{C}}$
defined for every $\mathbb{C}$-scheme $T$ by $(\tau^{*}F)(T)=F(\overline{T})$
which satisfies the identity $\tau^{*}\circ\tau^{*}=\mathrm{id}_{\mathcal{F}_{\mathbb{C}}}$
and commutes with fibre products. 

\subsubsection{\label{subsec:Real-structures-Funct}Real structures and real functors}
\begin{definition}
\label{def:Funct-Real-struct}A \emph{real functor} on the category
$(\mathrm{Sch}_{/\mathbb{C}})$ is a pair consisting of a functor
$F\in\mathrm{Ob}(\mathcal{F}_{\mathbb{C}})$ and a morphism of functors
$\alpha\colon F\to\tau^{*}F$ such that $\tau^{*}\alpha\circ\alpha=\mathrm{id}_{F}$,
called a \emph{real structure} on $F$. A \emph{real morphism} between
real functors $(F,\alpha)$ and $(F',\alpha')$ is a morphism of functors
$g\colon F\to F'$ such that $\alpha'\circ g=\tau^{*}g\circ\alpha$. 
\end{definition}

\begin{example}
\label{exa:Canonical-Real-Struct}Since $\nu=\nu\circ\tau$, for every
functor $H\in\mathrm{Ob}(\mathcal{F}_{\mathbb{R}})$, the functor
$H_{\mathbb{C}}=\nu^{*}H\in\mathrm{Ob}(\mathcal{F}_{\mathbb{C}})$
admits a canonical real structure $\alpha_{H}\colon H_{\mathbb{C}}\to\tau^{*}H_{\mathbb{C}}$
defined by the maps $\alpha_{H}(T)=\mathrm{id}_{H(\nu_{*}T)}$ for
every $\mathbb{C}$-scheme $T$. 
\end{example}

A real structure $\alpha$ on $F\in\mathrm{Ob}(\mathcal{F}_{\mathbb{C}})$
determines the following subfunctor $F/\alpha\in\mathrm{Ob}(\mathcal{F}_{\mathbb{R}})$
of $\nu_{*}F$ : 

$\bullet$ For every $\mathbb{R}$-scheme $S$, $(F/\alpha)(S)$ is
the subset of $(\nu_{*}F)(S)=F(S_{\mathbb{C}})=\mathrm{Hom}_{\mathcal{F}_{\mathbb{C}}}(\underline{S}_{\mathbb{C}},F)$
consisting of real morphisms $(\underline{S}_{\mathbb{C}},\alpha_{\underline{S}})\to(F,\alpha)$. 

$\bullet$ For every morphism of $\mathbb{R}$-schemes $h\colon S'\to S$,
the map $(F/\alpha)(h)\colon (F/\alpha)(S)\to(F/\alpha)(S')$ is the restriction
of $(\nu_{*}F)(h)$ to $(F/\alpha)(S)$, given by the composition
with the real morphism $h_{\mathbb{C}}\colon (\underline{S'}_{\mathbb{C}},\alpha_{\underline{S'}})\to(\underline{S}_{\mathbb{C}},\alpha_{\underline{S}})$.
\\

The morphism $F/\alpha\to\nu_{*}F$ determines by adjunction between
$\nu^{*}$ and $\nu_{*}$ a morphism $\psi\colon \nu^{*}(F/\alpha)\to F$
which is a real isomorphism between $(\nu^{*}(F/\alpha),\alpha_{F/\alpha})$
and $(F,\alpha)$. with the universal property that for every $H\in\mathrm{Ob}(\mathcal{F}_{\mathbb{R}})$
and every real morphism $f\colon (F,\alpha)\to(H_{\mathbb{C}},\alpha_{H})$,
there exists a unique morphism $f'\colon F/\alpha\to H$ such that $f=f'\circ\psi^{-1}$.
Moreover, for a real morphism $f\colon (F,\alpha)\to(F',\alpha')$, it follows
from the definition that for every $\mathbb{R}$-scheme $S$, the
image of $(F/\alpha)(S)\subset(\nu_{*}F)(S)$ by $(\nu_{*}f)(S)\colon (\nu_{*}F)(S)\to(\nu_{*}F')(S)$
is contained in $(F'/\alpha')(S)$ and hence, that $f$ induces a
morphism $\nu_{*}f\colon F/\alpha\to F'/\alpha'$. By construction of the
functor $F/\alpha$, we obtain the following: 

\begin{proposition}
\label{lem:Descent-on-functors}The association $(F,\alpha)\mapsto F/\alpha$
and $(f\colon (F,\alpha)\to(F',\alpha'))\mapsto(\nu_{*}f\colon F/\alpha\to F'/\alpha')$
is an equivalence between the category of real functors on $(\mathrm{Sch}_{/\mathbb{C}})$
and the category $\mathcal{F}_{\mathbb{R}}$, with quasi-inverse given
by the association $H\mapsto(\nu^{*}H,\alpha_{H})$ and $(h\colon H\to H')\mapsto(\nu^{*}h\colon (\nu^{*}H,\alpha_{H})\to(\nu^{*}H',\alpha_{H'})$. 
\end{proposition}

\begin{example}
\label{exa:Weil-Restriction}Given a functor $F\in\mathrm{Ob}(\mathcal{F}_{\mathbb{C}})$,
the functor $\nu_{*}F\in\mathrm{Ob}(\mathcal{F}_{\mathbb{R}})$ corresponds
under the equivalence of categories of Proposition~\ref{lem:Descent-on-functors}
to the functor $(\nu_{*}F)_{\mathbb{C}}=\nu^{*}\nu_{*}F\in\mathrm{Ob}(\mathcal{F}_{\mathbb{C}})$
defined for every $\mathbb{C}$-scheme $T$ by 
\[
(\nu_{*}F)_{\mathbb{C}}(T)=F(\nu^{*}\nu_{*}T)=F(T\sqcup\overline{T}),
\]
endowed with the real structure $\alpha_{\nu_{*}F}\colon (\nu_{*}F)_{\mathbb{C}}\to\tau^{*}(\nu_{*}F)_{\mathbb{C}}$
defined by $\alpha_{\nu_{*}F}(T)=\mathrm{id}_{F(T\sqcup\overline{T})}$.
When $F$ commutes with finite colimits, the canonical maps $F(T\sqcup\overline{T})\to F(T)\times F(\overline{T})$
define a real isomorphism $((\nu_{*}F)_{\mathbb{C}},\alpha_{\nu_{*}F})\to(F\times_{\underline{\C}}\tau^{*}F,\mathrm{e})$
for the real structure $\mathrm{e}\colon F\times_{\underline{\C}}\tau^{*}F\to\tau^{*}(F\times_{\underline{\C}}\tau^{*}F)=\tau^{*}F\times_{\underline{\C}} F$
exchanging the two factors. 
\end{example}

\begin{definition}\label{def:real group functor} A \emph{real group functor} is a pair consisting of a
group-valued functor $G\in \mathrm{Ob}(\mathcal{F}_\C)$ with structure maps $m\colon G\times_{\underline{\C}} G\to G$, $e\colon \underline{\mathrm{Spec}(\mathbb{C})}\to G$
and $\iota\colon G\to G$ and a real structure $\alpha_G$ on $G$ such that
all the morphisms $m\colon (G\times G,\alpha_{G}\times\alpha_{G})\to(G,\alpha_{G})$,
$e\colon (\underline{\mathrm{Spec}(\mathbb{C})},\alpha_{\underline{\mathrm{Spec}(\mathbb{R})}})\to(G,\alpha_{G})$
and $\iota\colon (G,\alpha_{G})\to(G,\alpha_{G})$ are real. 
\end{definition}
It is readily verified that the equivalence of categories of Proposition~\ref{lem:Descent-on-functors} induces
an equivalence between the category of real group functors on $(\mathrm{Sch}_{/\mathbb{C}})$
and the category of contravariant group-valued functors on $(\mathrm{Sch}_{/\mathbb{R}})$.

\begin{example}
\label{exa:AutGrpFunctor} For every $\R$-scheme $X$ with automorphism group functor $\mathfrak{G}=\underline{\mathrm{Aut}}_\R(X)$, the group functor $\mathfrak{G}_\C=\nu^*\mathfrak{G}$ endowed with the canonical real structure $\alpha_{\mathfrak{G}}$ on the the underlying set-valued functor of $\nu^*\mathfrak{G}$ is a real group functor. On the other hand, the collection of maps defined for every $\R$-scheme $S$ by associating to an $S$-automorphism $\Phi\colon S\times_\R X\to S\times_R X$ of $X$ the $S_\C$-automorphism $\Phi_\C\colon S_\C\times_\C X_\C\to S_\C\times_\C X_\C$ of $X_\C$ obtained by base change with the morphism $\nu\colon\mathrm{Spec}(\C)\to \mathrm{Spec}(\R)$ determines a canonical morphism  $\mathfrak{G}\to \nu_*\underline{\mathrm{Aut}}_\C(X_\C)$ whose adjoint morphism $\mathfrak{G}_\C\to \underline{\mathrm{Aut}}_\C(X_\C)$ is an isomorphism of group functors on the category of $\C$-schemes. We will henceforth abuse the notation and consider $\alpha_\mathfrak{G}$ as a real structure on $\underline{\mathrm{Aut}}_\C(X_\C)$ making it into a real group functor such that $\underline{\mathrm{Aut}}_\C(X_\C)/\alpha_\mathfrak{G}\cong \underline{\mathrm{Aut}}_\R(X)$ as group functors on the category of $\R$-schemes. 
\end{example}

\subsubsection{\label{subsec:Classical-real-structures}Classical real structures
and representability of the associated functors}

By Yoneda's Lemma every real structure $\alpha\colon \underline{Y}\to\tau^{*}\underline{Y}$
on the functor of points of a $\mathbb{C}$-scheme $Y$ is determined
by an isomorphism of $\mathbb{C}$-schemes 
\[
\sigma=\alpha(Y)(\mathrm{id}_{Y})\in\underline{Y}(\overline{Y})=\mathrm{Hom}_{\mathbb{C}}(\overline{Y},Y)
\]
that satisfies the identity $\sigma\circ\tau_{*}\sigma=\mathrm{id}_{Y}$,
classically called a \emph{real structure on} $Y$, the morphism $\alpha$
being then equal to that $\alpha_{\sigma}$ defined for every $\mathbb{C}$-scheme
$T$ by $\alpha_{\sigma}(T)\colon \underline{Y}(T)\to(\tau^{*}\underline{Y})(T)=\underline{Y}(\overline{T})$,
$g\mapsto\sigma\circ\tau_{*}g$. In particular, for an $\mathbb{R}$-scheme
$X$, the canonical real structure $\alpha_{\underline{X}}$ on $\underline{X}_{\mathbb{C}}$
of Example~\ref{exa:Canonical-Real-Struct} corresponds to the canonical
real structure 
\begin{equation}
\sigma_{X}\coloneqq\alpha_{\underline{X}}(X_{\mathbb{C}})(\mathrm{id}_{\mathbb{C}})=\mathrm{id}_{X}\times\tau\in\mathrm{Hom}_{\mathbb{C}}(\overline{X_{\mathbb{C}}},X_{\mathbb{C}}),\label{eq:Canonical-Real-Struct-Scheme}
\end{equation}
on $X_{\mathbb{C}}$. In the same way, a real morphism $(\underline{Y},\alpha_{\sigma})\to(\underline{Y}',\alpha_{\sigma'})$
between the real functors associated to $\mathbb{C}$-schemes with
real structures $(Y,\sigma)$ and $(Y',\sigma')$ is fully determined
by a \emph{real morphism} $(Y,\sigma)\to(Y',\sigma')$, that is, a
morphism of $\mathbb{C}$-schemes $g\colon Y\to Y'$ such that $g\circ\sigma=\sigma'\circ\tau_{*}g$. 

\begin{example}
\label{exa:real-struct-aut} For every $\R$-scheme $X$ with automorphism group functor $\mathfrak{G}=\underline{\mathrm{Aut}}_\R(X)$, the canonical real structure $\alpha_\mathfrak{G}$ on $\underline{\mathrm{Aut}}_\C(X_\C)$ (see Example \ref{exa:AutGrpFunctor}) is given for every $\C$-scheme $T$ by the map \[\alpha_{\mathfrak{G}}(T)\colon\underline{\mathrm{Aut}}_\C(X_\C)(T)\to (\tau^*\underline{\mathrm{Aut}}_\C(X_\C))(T)=\underline{\mathrm{Aut}}_\C(X_\C)(\overline{T})\]
which associates to a $T$-automorphism $\Psi\colon T\times_\C X_\C \to T\times_\C X_\C$ of $X_\C$ the $\overline{T}$-automorphism \[\alpha_{\mathfrak{G}}(T)(\Psi)=(\mathrm{id}_{\overline{T}}\times \sigma_X)\circ \tau_*\Psi\circ  (\mathrm{id}_{\overline{T}}\times \sigma_X^{-1})\]
of $X_\C$, where $\sigma_X\colon\overline{X_\C}\to X_\C$ is the canonical real structure on $X_\C$.
\end{example}

The functor $\underline{Y}/\alpha_{\sigma}\in\mathrm{Ob}(\mathcal{F}_{\mathbb{R}})$
associated to a real structure $\sigma$ on a $\mathbb{C}$-scheme
$Y$ is not always representable by an $\mathbb{R}$-scheme (see Example
\ref{exa:Non-rep-form} below). Nevertheless, we have the following
criterion: 
\begin{proposition}[{\cite[VIII-Corollaire 7.6]{SGAI}}]
\label{prop:Effective-descent}For a $\mathbb{C}$-scheme with real
structure $(Y,\sigma)$, the functor $\underline{Y}/\alpha_{\sigma}$
is representable by an $\mathbb{R}$-scheme if and only if $Y$ is
covered by $\sigma$-stable affine open subsets.\footnote{Here we interpret $\sigma\colon \overline{Y}\to Y$ as a morphism of schemes
$Y\to Y$ after forgetting the specific $\mathbb{C}$-scheme structures
on $\overline{Y}$ and $Y$.} 

In particular, every real structure $\sigma$ on a quasi-projective
$\mathbb{C}$-scheme $f\colon Y\to\mathrm{Spec}(\mathbb{C})$ uniquely determines
a quasi-projective $\mathbb{R}$-scheme $h\colon X=Y/\sigma\to\mathrm{Spec}(\mathbb{R})$
such that the following diagram is Cartesian 

\begin{equation}\label{eq:descent-diagram}
\xymatrix{ (Y,\sigma)\cong (X_\mathbb{C},\sigma_X) \ar[d]_{f} \ar[r] & X=Y/\sigma \ar[d]^{h} \\ \mathrm{Spec}(\mathbb{C}) \ar[r]^{\nu} & \mathrm{Spec}(\mathbb{R}).}
\end{equation}
\end{proposition}

\begin{example}
\label{exa:Real-locus}For a $\mathbb{C}$-scheme with real structure
$(Y,\sigma)$, the elements of the set $(\underline{Y}/\alpha_{\sigma})(\mathrm{Spec}(\mathbb{R}))$
correspond by definition to $\mathbb{C}$-rational points $y\colon \mathrm{Spec}(\mathbb{C})\to Y$
which are real morphism of $\mathbb{C}$-schemes $y\colon (\mathrm{Spec}(\mathbb{C}),\tau)\to(Y,\sigma)$.
We call the set $Y(\mathbb{C})^{\sigma}$ of such points the \emph{real
locus} $(Y,\sigma)$ and refer its elements to as the \emph{real points
of $(Y,\sigma)$}. Note that when the functor $\underline{Y}/\alpha_{\sigma}$ is representable by an $\R$-scheme $X$, the set $Y(\mathbb{C})^{\sigma}$ 
then identifies by construction with the set $X(\R)\coloneqq X(\mathrm{Spec(\R)})$ of $\R$-rational points of $X$.
\end{example}

\begin{example}
\label{exa:WeilRestriction-2} For a $\mathbb{C}$-scheme $Y$ such
that $Y\times_{\mathbb{C}}\overline{Y}$ admits a covering by affine
open subsets which are stable under the real structure exchanging
the two factors, it follows from Proposition~\ref{prop:Effective-descent}
that the functor $\nu_{*}\underline{Y}\cong(\underline{Y}\times_{\underline{\C}}\underline{\overline{Y}})/\mathrm{e}$
of Example~\ref{exa:Weil-Restriction} is representable by an $\mathbb{R}$-scheme
$R_{\mathbb{C}/\mathbb{R}}Y$, called the \emph{Weil restriction}
of $Y$, see e.g.~\cite[Chapter 7.6]{BLR}. 
\end{example}

\begin{example}
\label{exa:Non-rep-form}A well-known example of an $\mathbb{R}$-scheme
whose real forms are not all representable is the real affine line
with a double origin, that is, the non-separated $\mathbb{R}$-scheme
$X$ obtained by gluing two copies $X_{\pm}$ of $\mathbb{A}_{\mathbb{R}}^{1}$
by the identity outside their respective origins $o_{\pm}$. Namely,
consider the real structure $\sigma$ on $X_{\mathbb{C}}$ defined
as the composition of the involution $\tau_{*}\Psi$ exchanging the
two open subsets $\overline{X}_{\pm,\mathbb{C}}\cong\mathbb{A}_{\mathbb{C}}^{1}$
of the covering of $\overline{X}_{\mathbb{C}}$ with the canonical
real structure $\sigma_{X}$. Every $\sigma$-stable open neighbourhood
of $o_{+,\mathbb{C}}$ in $X_{\mathbb{C}}$ contains contains $o_{-,\mathbb{C}}$
hence is a non-separated scheme -in particular a non-affine scheme-
from which it follows that $\underline{X_{\mathbb{C}}}/\alpha_{\sigma}$
is not representable by an $\mathbb{R}$-scheme.
\end{example}

\section{Relative real forms }

We introduce a notion of relative real form of an $\mathbb{R}$-scheme
$X$ over a base $\mathbb{R}$-scheme $S$ and construct for every
$\mathbb{R}$-scheme $X$ a morphism of functors $u\colon \mathfrak{U}_{X}\to\mathfrak{Z}$
in $\mathcal{F}_{\mathbb{R}}$ such that every $S$-form of $X$ over
an $\mathbb{R}$-scheme $S$ is induced from $u\colon \mathfrak{U}_{X}\to\mathfrak{Z}$
by base change by a morphism $h\colon \underline{S}\to\mathfrak{Z}$.
We relate $\mathbb{R}$-valued points of these functors to the classical
Borel-Serre Galois cohomology classification of real forms of $X$
and in the case where $X$ is quasi-projective with representable
automorphism group functor $\underline{\mathrm{Aut}}_{\mathbb{R}}(X)$,
we establish that $\mathfrak{Z}$ and $\mathfrak{U}_{X}$ are
both representable. 

\subsection{\label{subsec:familie-basic}Real forms of functors and relative forms of real schemes}
\subsubsection{\label{sub:realform-funct}Real forms of functors on real schemes}
\begin{definition}
\label{def:RealForm-RScheme}A \emph{real form} of a functor $H\in\mathrm{Ob}(\mathcal{F}_{\mathbb{R}})$
is a pair $(H',\theta)$ consisting of a functor $H'\in\mathrm{Ob}(\mathcal{F}_{\mathbb{R}})$
and an isomorphism $\theta\colon H'_{\mathbb{C}}\to H_{\mathbb{C}}$ in
$\mathcal{F}_{\mathbb{C}}$. An isomorphism between real forms $(H',\theta)$
and $(H'',\theta')$ of $H$ is an isomorphism $\xi\colon H''\to H'$ such
that $\theta'=\theta\circ\xi_{\mathbb{C}}$. A real form of an $\mathbb{R}$-scheme
is a real form of its functor of points.
\end{definition}

As an illustration of the above notion and a preparation for the next subsections, we now introduce and describe some properties of a natural real form of every group-valued functor $G\in \mathrm{Ob}(\mathcal{F}_\R)$. Denote by $m$, $e$ and $\iota$ the structure maps of the group functor $G$ and let $\alpha_G\colon G_\C\to \tau^*G_\C$ be the canonical real structure on $G_\C=\nu^*G$ as in Definition~\ref{def:real group functor}. Then $\nu_*G_\C\in \mathrm{Ob}(\mathcal{F}_\R)$ is a group-valued functor with structure maps $\tilde{m}$, $\tilde{e}$ and $\tilde{\iota}$ corresponding respectively under the adjunction between $\nu^*$ and $\nu_*$ to the structure maps $m_\C$, $e_\C$ and $\iota_\C$ of $G_\C$. In the rest of this subsection, we also use the shortand notation $\tilde{G}$ for $\nu_*G_\C$. Recall by Example~\ref{exa:Weil-Restriction} that when  $G_\C$ commutes with finite colimits, $\tilde{G}$ corresponds under the equivalence of categories of Proposition~\ref{lem:Descent-on-functors} to the functor $G_\C\times_{\underline{\C}} \tau^*G_\C$ endowed with the real structure $\mathrm{e}$ exchanging the two factors. 

By definition, for every $\R$-scheme $S$, we have $\tilde{G}(S)=G_\C(S_\C)=\mathrm{Hom}_{\mathcal{F}_\C}(\underline{S}_\C,G_\C)$. The collection of involutions  $c(S)\colon \tilde{G}(S)\to \tilde{G}(S)$ given for every $\R$-scheme $S$ by mapping an element $f\in \mathrm{Hom}_{\mathcal{F}_\C}(\underline{S}_\C,G_\C)$ to its \emph{conjugate} $c(S)(f)\coloneqq\tau^*\alpha_G\circ \tau^*f\circ \alpha_{\underline{S}}$, where $\alpha_{\underline{S}}$ is the canonical real structure on $\underline{S}_\C$, determines a group functor involution $c\colon \tilde{G}\to \tilde{G}$ of $\tilde{G}$.  The composition \[\tilde{\gamma}=\tilde{m}\circ (\mathrm{pr}_1, {\tilde{m}}\circ\mathrm{pr}_{23})\circ (c\circ \mathrm{pr}_1,\mathrm{pr}_2,\tilde{\iota}\circ \mathrm{pr}_1) \colon \tilde{G}\times_{\underline{\R}} \tilde{G}\to \tilde{G}\times_{\underline{\R}} \tilde{G}\times_{\underline{\R}} \tilde{G}\to \tilde{G}\times_{\underline{\R}} \tilde{G} \to \tilde{G}, \; ``\,(g,h)\mapsto c(g)\cdot h\cdot g^{-1}\, "\]
defines in turn an action of the group functor $\tilde{G}$ on the underlying set-valued functor of $\tilde{G}$; here, $\mathrm{pr}_i$ denotes the projection onto the $i$th factor, and $\mathrm{pr}_{23}$ the projection onto the last two factors, and the symbol ``$\cdot$'' is a shorthand notation for the composition law in the group $\tilde{G}(S)$.

The identity $\alpha_{G}\circ\iota_{\mathbb{C}}=\tau^{*}\iota_{\mathbb{C}}\circ\alpha_{G}$ implies that $\alpha_{G}\circ\iota_{\mathbb{C}}$ is a real structure on $G_{\mathbb{C}}$. By $\S$~\ref{subsec:Real-structures-Funct}, the functor corresponding to the real functor $(G_{\mathbb{C}},\alpha_{G}\circ\iota_{\mathbb{C}})$
under the equivalence of categories of Proposition~\ref{lem:Descent-on-functors} equals 
the subfunctor 
\begin{equation}\label{eq:Z_G} j\colon\mathfrak{Z}_G\coloneqq G_{\mathbb{C}}/(\alpha_{G}\circ\iota_{\mathbb{C}})\hookrightarrow \nu_{*}G_{\mathbb{C}} \end{equation}
of the underlying set-valued finctor of $\nu_*G_\C$ whose $S$-points, where $S$ is any $\mathbb{R}$-scheme,
are the real morphism $(\underline{S}_{\mathbb{C}},\alpha_{\underline{S}})\to(G_{\mathbb{C}},\alpha_{G}\circ\iota_{\mathbb{C}})$.
The morphism $\theta\colon\mathfrak{Z}_G\to G_{\mathbb{C}}$ in $\mathcal{F}_{\mathbb{C}}$
corresponding to $j$ under the adjunction between $\nu^{*}$ and
$\nu_{*}$ is an isomorphism, which makes the pair $(\mathfrak{Z}_G,\theta)$
a real form of the underlying set-valued functor of $G$.  

\begin{lemma}\label{lem:WeilRest-Action-Z_G} The subfunctor $j\colon\mathfrak{Z}_G\hookrightarrow \nu_*G_\C$ is stable under the action $\tilde{\gamma}$ of $\nu_*G_\C$. 
\end{lemma}
\begin{proof}
By definition, for an $\R$-scheme $S$, $\mathfrak{Z}_G(S)\subset \tilde{G}(S)=\mathrm{Hom}_{\mathcal{F}_\C}(\underline{S}_\C,G_\C)$ consists of morphisms $h\colon\underline{S}_\C\to G_\C$ such that $\alpha_G\circ \iota_\C\circ h=\tau^*h\circ \alpha_{\underline{S}}$. Given an element $g\colon\underline{S}_\C\to G_\C$ of $\tilde{G}(S)$, we have by definition $\gamma(S)(g,h)=(c\circ g)\cdot h\cdot (\iota_\C\circ g)\; \reflectbox{$\coloneqq$}\; h'$, where $c\circ g=\tau^*\alpha_G\circ \tau^*g\circ \alpha_{\underline{S}}$. It follows that  
\[\begin{array}{rcl} 
\alpha_{G}\circ\iota_{\C}\circ h' &= & \alpha_{G}\circ\iota_{\C}\circ((c\circ g)\cdot h\cdot(\iota_{\C}\circ g)) = (\alpha_{G}\circ g)\cdot(\alpha_{G}\circ\iota_{\C}\circ h)\cdot(\alpha_{G}\circ\iota_{\C}\circ c \circ g) \\ &= & (\alpha_{G}\circ g)\cdot(\tau^{*}h\circ\alpha_{\underline{S}})\cdot(\tau^{*}\iota_{\C}\circ\alpha_{G}\circ c\circ g) = (\tau^{*}(c\circ g)\circ\alpha_{\underline{S}})\cdot(\tau^{*}h\circ\alpha_{\underline{S}})\cdot(\tau^{*}\iota_{\C}\circ\tau^{*}g\circ\alpha_{\underline{S}}) \\ & = &\tau^{*}h'\circ\alpha_{\underline{S}},
\end{array}\]
which shows that $\tilde{\gamma}(S)$ maps the subset $\mathfrak{Z}_G(S)\subset \tilde{G}(S)$ onto itself. 
\end{proof}

We thus obtain for every group-valued functor $G\in \mathrm{Ob}(\mathcal{F}_\R)$ a pair consisting of a real form $\mathfrak{Z}_G=G_\C/(\alpha_G\circ \iota_\C)$ of the underlying set-valued functor of $G$ and an action 
\begin{equation}\label{eq:Weil-act-Z_G} 
 \gamma\colon \nu_*G_\C\times_{\underline{\R}} \mathfrak{Z}_G \to \mathfrak{Z}_G
 \end{equation}
 of the group functor $\nu_*G_\C\in \mathrm{Ob}(\mathcal{F}_\R)$ on $\mathfrak{Z}_G$ induced by the action $\tilde{\gamma}$ of $\nu_*G_\C$ on itself. 

\begin{example}\label{exa:PGL2} Consider the group functor $G=\underline{\mathrm{PGL}}_{2,\R}=\underline{\mathrm{PSL}}_{2,\R}$. The group functor $G_\C$ is representable by the affine group $\C$-scheme 
$\mathrm{PSL}_{2,\C}$ defined as the quotient affine group $\C$-scheme 
$\mathrm{SL}_{2,\C}=\{\left(\begin{smallmatrix} a & b \\ c & d \end{smallmatrix}\right) | \; ad-bc=1\}\subset \mathbb{A}^4_\C$ by the involution $(a,b,c,d)\mapsto (-a,-b,-c,-d)$.
The restriction to $\mathrm{SL}_{2,\C}$ of the quotient projection $\mathbb{A}^4_\C\setminus \{0\}\to \mathbb{P}^3_\C$, $(a,b,c,d)\mapsto [a:b:c:d]$ identifies in turn $\mathrm{PSL}_{2,\C}=\mathrm{SL}_{2,\C}/\mu_2$ with the complement $V$ of the smooth quadric surface $Q=\{ad-bc=0\}$ in $\mathbb{P}^3_\C$. The real structure $\alpha_G\circ \iota_\C$ on $G_\C$ equals that associated to the restriction $\sigma$ on $V$ of the real structure on $\mathbb{P}^3_\C$ defined as the composition of the involution $[a:b:c:d]\mapsto [d:-b:-c:a]$ with the canonical real structure $\sigma_{\mathbb{P}^3_\R}$. 
The linear coordinate change 
\[\xi\colon\mathbb{A}_{\mathbb{C}}^{4}\ni(a,b,c,d)\mapsto(x,y,z,t)=(\tfrac{1}{2}(d+a),\tfrac{1}{2i}(d-a), \tfrac{1}{2i}(c+b),\tfrac{1}{2i}(c-b))\] induces a real isomorphism between $(V,\sigma)$ and the complexificaton of the complement of the smooth quadric surface  $Q_0=\{x^2+y^2+z^2-t^2=0\}\subset \mathbb{P}^3_\R$ endowed with its canonical real structure. The functor $\mathfrak{Z}_G$ is thus represented by the affine $\R$-scheme $Z=\mathbb{P}^3_\R\setminus Q_0$.

Letting $q_0(x,y,z,t)=x^2+y^2+z^2-t^2$, the real locus $Z(\R)$ of $Z$ endowed with its Euclidean topology is the image by the quotient projection $\R^4\setminus\{0\}\to \R\mathbb{P}^3=(\mathbb{R}^{4}\setminus\{0\})/\mathbb{R}^*$ of the complement of the ``light'' cone $\mathcal{C}=\{q_0=0\}\subset \R^4$. It is a smooth manifold consisting of two connected components: one $Z_-(\mathbb{R})$, diffeomorphic to $\R^3$, which is the image of the hyperboloid $\mathcal{H}_-=\{q_0=-1, t\in \R_{\leq 0}\}\sqcup\{q_0=-1,t\in R_{\geq 0}\}\cong\mathbb{R}^{3}\sqcup\mathbb{R}^{3}$ and another one $Z_+(\R)$, diffeomorphic to the total space of the unique non-trivial real line bundle over $\R\mathbb{P}^2$, which is the image of the connected hyperboloid $\mathcal{H}_{+}=\{q_0=1\}\cong S^{2}\times\mathbb{R}$.
The group functor $\nu_*G_\C$ is representable by the Weil restriction $R_{\C/\R}\mathrm{PSL}_{2,\C}$ whose real locus endowed with the Euclidean topology is the Lie group $\mathrm{PSL}_2(\C)$. The associated action $\gamma(\R)\colon(R_{\C/\R}\mathrm{PSL}_{2,\C})(\R)\times Z(\R)\to Z(\R)$ is induced through the quotient map $\R^4\setminus\{0\}\to \R\mathbb{P}^3$ 
by the linear action on  $\R^4$ of $\mathrm{PSL}_2(\C)$ identified using the coordinate change above to the restricted Lorentz group $\mathrm{SO}^{+}(1,3)=\mathrm{SO}^{+}(q_0)$. Namely, an element of $\mathrm{PSL}_2(\C)$ represented by a matrix $M=\left(\begin{smallmatrix} a & b \\ c & d \end{smallmatrix}\right)\in \mathrm{SL}_2(\C)$ acts on $\R^4$ identified with the vector space of Hermitian matrices $H=\left(\begin{smallmatrix} t+z & x-iy \\x+iy & t-z\end{smallmatrix}\right)$ by $H\mapsto\,\!^t\overline{M}^{-1}HM^{-1}$ where $\;\!^t\overline{M}$ is the conjugate transpose of the matrix $M$. The $(R_{\C/\R}\mathrm{PSL}_{2,\C})(\R)$-orbits on $Z(\R)$ are the respective images $Z_-(\R)$ and $Z_+(\R)$ of the orbits $\{q_0=-1,\; t\in \R_{\geq 0}\}$ and $\{q_0=1\}$ of the induced $\mathrm{SO}^{+}(q_0)$-action on $\R^4\setminus \mathcal{C}$  by the projection  $\R^4\setminus\{0\}\to \R\mathbb{P}^3$. 
\end{example}

\subsubsection{Relative real forms of real schemes}

\begin{definition}
\label{def:Family-RForms} A \emph{relative form} of an $\mathbb{R}$-scheme
$X$ over a functor $H\in\mathrm{Ob}(\mathcal{F}_{\mathbb{R}})$ is
a real form of the functor $H\times_{\underline{\mathbb{R}}}\underline{X}$
over $H$, that is, a pair $(f\colon \mathfrak{X}\to H,\theta)$ where $f\colon \mathfrak{X}\to H$
is a morphism in $\mathcal{F}_{\mathbb{R}}$ and where $\theta\colon \mathfrak{X}_{\mathbb{C}}\to(H\times_{\underline{\mathbb{R}}}\underline{X})_{\mathbb{C}}$
is an isomorphism in $\mathcal{F}_{\mathbb{C}}$ such that $f_{\mathbb{C}}=\mathrm{pr}_{\underline{S}_{\mathbb{C}}}\circ\theta$.
An isomorphism between two $H$-forms $(f\colon \mathcal{\mathfrak{X}}\to H,\theta)$
and $(f'\colon \mathcal{\mathfrak{X}}'\to H,\theta')$ of $X$ is an isomorphism
$\xi\colon \mathcal{\mathfrak{X}}'\to\mathcal{\mathfrak{X}}$ in $\mathcal{F}_{\mathbb{R}}$
such that $f'=f\circ\xi$ and $\theta'=\theta\circ\xi_{\mathbb{C}}$.
An \emph{$S$-form of $X$} over an $\mathbb{R}$-scheme $S$ is a relative
form of $X$ over the functor of points $\underline{S}$ of $S$. 
\end{definition}

The following lemma is a direct consequence of the fact that the functor
$\nu^{*}$ commutes with fibre products: 
\begin{lemma}
\label{lem:Base-change-Def}Let $(f\colon \mathcal{\mathfrak{X}}\to\underline{S},\theta)$
be an $S$-form of an $\mathbb{R}$-scheme $X$ and let $h\colon S'\to S$
be a morphism of $\mathbb{R}$-schemes. Then the fibre product $f_{S',h}=\mathrm{pr}_{1}\colon \mathfrak{X}_{S',h}=\underline{S}'\times_{\underline{h},\underline{S},f}\mathfrak{X}\to\underline{S'}$
in $\mathcal{F}_{\mathbb{R}}$ together with the isomorphism 
\[
\theta_{S',h}\coloneqq\mathrm{id}_{\underline{S'}_{\mathbb{C}}}\times_{\underline{S}_{\mathbb{C}}}\theta\colon \mathfrak{X}'_{\mathbb{C}}\cong\underline{S'}_{\mathbb{C}}\times_{\underline{S}_{\mathbb{C}}}\mathfrak{X}_{\mathbb{C}}\to\underline{S'}_{\mathbb{C}}\times_{\underline{S}_{\mathbb{C}}}(\underline{S}\times_{\underline{\mathbb{R}}}\underline{X})_{\mathbb{C}}\cong(\underline{S}'\times_{\underline{\mathbb{R}}}\underline{X})_{\mathbb{C}}
\]
is an $S'$-form of $X$, called the \emph{base change} of $(f\colon \mathcal{\mathfrak{X}}\to\underline{S},\theta)$
by $h\colon S'\to S$. 
\end{lemma}

\begin{example}
The base change of an $S$-form $(f\colon \mathcal{\mathfrak{X}}\to\underline{S},\theta)$
of an $\mathbb{R}$-scheme $X$ by an $\mathbb{R}$-rational point
$h\colon \mathrm{Spec}(\mathbb{R})\to S$ of $S$ is a real form $(\mathcal{\mathfrak{X}}_{\mathrm{Spec(\mathbb{R}),}h},\theta_{\mathrm{Spec(\mathbb{R}),}h})$
of $X$ in the sense of Definition~\ref{def:RealForm-RScheme}. 
\end{example}

\begin{rem}
A morphism of $\mathbb{R}$-schemes $f\colon \mathcal{X}\to S$ and an isomorphism
of $S_{\mathbb{C}}$-schemes $\theta\colon \mathcal{X}_{\mathbb{C}}\to S_{\mathbb{C}}\times_{\mathbb{C}}X_{\mathbb{C}}$
for some $\mathbb{R}$-scheme $X$ determines by Yoneda's Lemma an
$S$-form $(f\colon \underline{\mathcal{X}}\to\underline{S},\theta)$ of
$X$. In the sequel, we will often call the pair $(f\colon \mathcal{X}\to S,\theta)$
itself an $S$-form of $X$. 
\end{rem}

\begin{example}
\label{exa:G_m-over-S1} Let $f\colon L\to\mathbb{S}^{1}$ be the unique
non-trivial line bundle over the norm one torus 
\[
\mathbb{S}^{1}=\mathrm{Spec}(\mathbb{R}[x,y]/(x^{2}+y^{2}-1)).
\]
Since the Picard group of $\mathbb{S}_{\mathbb{C}}^{1}\cong\mathbb{G}_{m,\mathbb{C}}$
is trivial, $f_{\mathbb{C}}\colon L_{\mathbb{C}}\to\mathbb{S}_{\mathbb{C}}^{1}$
is a trivial line bundle, and the choice of a trivialization $\theta\colon L_{\mathbb{C}}\to\mathbb{S}_{\mathbb{C}}^{1}\times_{\mathbb{C}}\mathbb{A}_{\mathbb{C}}^{1}$
determines a non-trivial $\mathbb{S}^{1}$-form $(f\colon L\to\mathbb{S}^{1},\theta)$
of $\mathbb{A}_{\mathbb{R}}^{1}$. Removing the zero section of the
line bundle $L$ yields a non-trivial $\mathbb{G}_{m,\mathbb{R}}$-torsor
$f_{0}\colon L^{\star}\to\mathbb{S}^{1}$ which, together with the induced
isomorphism $\theta_{0}\colon L_{\mathbb{C}}^{\star}\to\mathbb{S}_{\mathbb{C}}^{1}\times_{\mathbb{C}}\mathbb{A}_{\mathbb{R}}^{1}\setminus\{0\}$,
is a non-trivial $\mathbb{S}^{1}$-form of $\mathbb{A}_{\mathbb{R}}^{1}\setminus\{0\}$. 
\end{example}

\begin{example}
\label{exa:DFM} The Picard group of the real affine $2$-sphere $\mathbb{S}^{2}=\mathrm{Spec}(\mathbb{R}[x,y,z]/(x^{2}+y^{2}+z^{2}-1))$
is trivial whereas the Picard group of its complexification $\mathbb{S}_{\mathbb{C}}^{2}$
is isomorphic to $\mathbb{Z}$, generated by the class of the line
bundle $L\to\mathbb{S}_{\mathbb{C}}^{2}$ corresponding to the Cartier
divisor $D=\{x-iy=1-z=0\}$. For every $n\geq0$, the Weil restriction
$R_{\mathbb{C}/\mathbb{R}}L^{n}$ of $L^{n}$ is a vector bundle of
rank $2$ over the Weil restriction $R_{\mathbb{C}/\mathbb{R}}\mathbb{S}_{\mathbb{C}}^{2}$
of $\mathbb{S}_{\mathbb{C}}^{2}$. The pullback of $R_{\mathbb{C}/\mathbb{R}}L^{n}$
by the canonical closed immersion $\mathbb{S}^{2}\hookrightarrow R_{\mathbb{C}/\mathbb{R}}\mathbb{S}_{\mathbb{C}}^{2}$
is then a vector bundle $f_{n}\colon E_{n}\to\mathbb{S}^{2}$ of rank $2$
over $\mathbb{S}^{2}$. By~\cite{DFMJ21}, the induced smooth real
vector bundle $f_{n}(\mathbb{R})\colon E_{n}(\mathbb{R})\to\mathbb{S}^{2}(\mathbb{R})=S^{2}\cong\mathbb{CP}^{1}$
between $E_{n}(\mathbb{R})$ and $\mathbb{S}^{2}(\mathbb{R})$ endowed
with their respective structure of smooth real manifold is isomorphic
to the underlying real vector bundle on the complex line bundle $\mathcal{O}_{\mathbb{CP}^{1}}(n)$
on $\mathbb{CP}^{1}$. In particular, the vector bundles $f_{n}\colon E_{n}\to\mathbb{S}^{2}$,
$n\geq1$, are non-trivial and pairwise non-isomorphic. On the other
hand, their complexifications $f_{n,\mathbb{C}}\colon E_{n,\mathbb{C}}\to\mathbb{S}_{\mathbb{C}}^{2}$
are all isomorphic to the trivial vector bundle of rank $2$ over
$\mathbb{S}_{\mathbb{C}}^{2}$. The choice of trivialisations $\theta_{n}\colon E_{n,\mathbb{C}}\to\mathbb{S}_{\mathbb{C}}^{2}\times_{\mathbb{C}}\mathbb{A}_{\mathbb{C}}^{2}$
thus determine a countable family $(f_{n}\colon E_{n}\to\mathbb{S}^{2},\theta_{n})$,
$n\geq0$, of pairwise non-isomorphic $\mathbb{S}^{2}$-forms of $\mathbb{A}_{\mathbb{R}}^{2}$. 
\end{example}

\begin{example}
\label{exa:Affine-conic-bundle}Let $\mathcal{X}\subset S\times_{\mathbb{R}}\mathbb{A}_{\mathbb{R}}^{2}=\mathrm{Spec}(\mathbb{R}[\lambda^{\pm1}][X,Y])$
be the subscheme with equation $X^{2}+Y^{2}=\lambda$, considered
as an affine conic bundle $f=\mathrm{pr}_{S}\colon \mathcal{X}\to S=\mathbb{A}_{\mathbb{R}}^{1}\setminus\{0\}$.
The isomorphism of $S_{\mathbb{C}}$-schemes 
\[
\theta\colon \mathcal{X}_{\mathbb{C}}\to S_{\mathbb{C}}\times_{\mathbb{C}}(\mathbb{A}_{\mathbb{C}}^{1}\setminus\{0\}),\,(\lambda,X,Y)\mapsto(\lambda,X+iY)
\]
makes $(f\colon \mathcal{X}\to S,\theta)$ an $S$-form of $\mathbb{A}_{\mathbb{R}}^{1}\setminus\{0\}$.
The scheme-theoretic fibre of $f$ over a point $\lambda\in S(\mathbb{R})=\mathbb{R}\setminus\{0\}$
is isomorphic either to the torus $\mathbb{S}^{1}$ if $\lambda>1$
or to the non-trivial $\mathbb{S}^{1}$-torsor $\mathrm{Spec}(\mathbb{R}[x,y]/(x^{2}+y^{2}+1))$
if $\lambda<0$. 
\end{example}

\begin{example}
In the same vein as in the previous example, consider the subvariety
\[
\mathcal{X}=\{x_{0}^{2}+\lambda x_{1}^{2}+\mu x_{2}^{2}=0\}\subset\mathbb{G}_{m,\mathbb{R}}^{2}\times_{\mathbb{R}}\mathbb{P}_{\mathbb{R}}^{2}=\mathrm{Proj}_{\mathbb{R}[\lambda^{\pm},\mu^{\pm1}]}(\mathbb{R}[\lambda^{\pm1},\mu^{\pm1}][x_{0},x_{1},x_{2}])
\]
as a smooth conic bundle $f=\mathrm{pr}_{2}|_{\mathcal{X}}\colon \mathcal{X}\to S=\mathbb{G}_{m,\mathbb{R}}^{2}$.
The fibre of $f$ over a point $(\lambda,\mu)\in S(\mathbb{R})$ is
isomorphic to $\mathbb{P}_{\mathbb{R}}^{1}$ if either $\lambda$
or $\mu$ is negative or to the non-trivial real form $\{x_{0}^{2}+x_{1}^{2}+x_{2}^{2}=0\}\subset\mathbb{P}_{\mathbb{R}}^{2}$
of $\mathbb{P}_{\mathbb{R}}^{1}$ when $\lambda$ and $\mu$ are both
positive. Since the generic fibre of $f_{\mathbb{C}}\colon \mathcal{X}_{\mathbb{C}}\to S_{\mathbb{C}}$
is a form of $\mathbb{P}_{\mathbb{C}(\lambda,\mu)}^{1}$ without $\mathbb{C}(\lambda,\mu)$-rational
point, $\mathcal{X}_{\mathbb{C}}$ is not $S_{\mathbb{C}}$-isomorphic
to $S_{\mathbb{C}}\times_{\mathbb{C}}\mathbb{P}_{\mathbb{C}}^{1}$,
in particular $f\colon \mathcal{X}\to S$ cannot be equipped with the structure
of an $S$-form of $\mathbb{P}_{\mathbb{R}}^{1}$. Nevertheless $f\colon \mathcal{X}\to S$
is a local $S$-form of $\mathbb{P}_{\mathbb{R}}^{1}$ with respect
to the \'etale topology on $S$: the base change of $f\colon \mathcal{X}\to S$
by the \'etale morphism $\tilde{S}=\mathbb{G}_{m,\mathbb{R}}^{2}\to\mathbb{G}_{m,\mathbb{R}}^{2}=S,(\xi,\mu)\mapsto(\lambda,\mu)=(\xi^{2},\mu)$
is an $\tilde{S}$-form of $\mathbb{P}_{\mathbb{R}}^{1}$. 
\end{example}

\subsection{\label{subsec:Tauto-form}The tautological complete relative form
of a real scheme }

In this subsection, we construct for every $\mathbb{R}$-scheme $X$ with automorphism group functor $\mathfrak{G}=\mathrm{\underline{Aut}}_\R(X)$ a relative form $(u\colon \mathfrak{U}_{X}\to\mathfrak{Z},\Theta_{u})$
of $X$ over the real form $\mathfrak{Z}\coloneqq\mathfrak{Z}_{\mathfrak{G}}$ of the underlying set-valued functor of $\mathfrak{G}$ described in $\S$~\ref{sub:realform-funct} which is\emph{ complete} in the following sense: for every $\mathbb{R}$-scheme $S$, every $S$-form $(f\colon \mathcal{\mathfrak{X}}\to\underline{S},\theta)$ of $X$ is isomorphic to the base change 
\[
(u_{h}=\mathrm{pr}_{\underline{S}}\colon \underline{S}\times_{h,\mathfrak{Z},u}\mathfrak{U}_{X}\to\underline{S},\Theta_{u,h}=\mathrm{id}_{\underline{S}_{\mathbb{C}}}\times_{\mathfrak{Z}_{\mathbb{C}}}\Theta_{u})
\]
for some morphism $h\colon \underline{S}\to\mathfrak{Z}$, where we
view $\mathrm{id}_{\underline{S}_{\mathbb{C}}}\times_{\mathfrak{Z}_{\mathbb{C}}}\Theta_{u}$
as an isomorphism 
\[
\Theta_{u,h}\colon (\underline{S}\times_{h,\mathfrak{Z},u}\mathfrak{U}_{X})_{\mathbb{C}}\cong\underline{S}_{\mathbb{C}}\times_{h_{\mathbb{C}},\mathfrak{Z}_{\mathbb{C}},u_{\mathbb{C}}}(\mathfrak{U}_{X})_{\mathbb{C}}\to\underline{S}_{\mathbb{C}}\times_{h_{\mathbb{C}},\mathfrak{Z}_{\mathbb{C}},\mathrm{pr}_{\mathfrak{Z}_{\mathbb{C}}}}(\mathfrak{Z}\times_{\underline{\mathbb{R}}}\underline{X})_{\mathbb{C}}\cong(\underline{S}\times_{\underline{\mathbb{R}}}\underline{X})_{\mathbb{C}}.
\]

\subsubsection{\label{subsec:Functor-Z_X} Tautological parameter space and moduli}

Let $X$ be an $\mathbb{R}$-scheme, let $\mathfrak{G}=\underline{\mathrm{Aut}}_{\mathbb{R}}(X)$
be its automorphism group functor and let $\mathfrak{G}_{\mathbb{C}}=\underline{\mathrm{Aut}}_{\mathbb{C}}(X_{\mathbb{C}})$.
Proposition~\ref{lem:Descent-on-functors} induces a one-to-one correspondence
between $S$-forms $(f\colon \mathfrak{X}\to\underline{S},\theta)$ of $X$
and real structures $\alpha_{(\mathfrak{X},\theta)}=\tau^{*}\theta\circ\alpha_{\mathfrak{X}}\circ\theta^{-1}$
on $(\underline{S}\times_{\underline{\mathbb{R}}}\underline{X})_{\mathbb{C}}$,
where $\alpha_{\mathfrak{X}}$ is the canonical real structure on
$\mathfrak{X}_{\mathbb{C}}$. Letting $\alpha_{\underline{S}\times\underline{X}}=\alpha_{\underline{S}}\times\alpha_{\underline{X}}$
be the canonical real structure on $(\underline{S}\times_{\underline{\mathbb{R}}}\underline{X})_{\mathbb{C}}$,
one has $\alpha_{(\mathfrak{X},\theta)}=\alpha_{\underline{S}\times\underline{X}}\circ\Psi$
for some $S_{\mathbb{C}}$-automorphism $\Psi\in\mathfrak{G}_{\mathbb{C}}(S_{\mathbb{C}})=\nu_{*}\mathfrak{G}_{\mathbb{C}}(S)$
of $S_{\mathbb{C}}\times_{\mathbb{C}}X_{\mathbb{C}}$ which satisfies
the cocycle identity
\begin{equation}
(\tau^{*}\alpha_{\underline{S}\times\underline{X}}\circ\tau^{*}\Psi\circ\alpha_{\underline{S}\times\underline{X}})\circ\Psi=\mathrm{id}_{(\underline{S}\times X)_{\mathbb{C}}}.\label{eq:Real-cocycle-cond}
\end{equation}
Furthermore, two $S$-forms $(f\colon \mathfrak{X}\to\underline{S},\theta)$
and $(f'\colon \mathfrak{X}'\to\underline{S},\theta')$ are isomorphic if
and only if their respective associated $S_{\mathbb{C}}$-automorphisms
$\Psi=\tau^{*}\alpha_{\underline{S}\times\underline{X}}\circ\alpha_{(\mathfrak{X},\theta)}$
and $\Psi'=\tau^{*}\alpha_{\underline{S}\times\underline{X}}\circ\alpha_{(\mathfrak{X}',\theta')}$
satisfy the relation 
\begin{equation}
\Psi'=(\tau^{*}\alpha_{\underline{S}\times\underline{X}}\circ\tau^{*}\Phi\circ\alpha_{\underline{S}\times\underline{X}})\circ\Psi\circ\Phi^{-1}\label{eq:Real-conjugacy-cond}
\end{equation}
for some $S_{\mathbb{C}}$-automorphism $\Phi$ of $S_{\mathbb{C}}\times_{\mathbb{C}}X_{\mathbb{C}}$.
The above correspondence between $S$-forms and $S_{\mathbb{C}}$-automorphisms
$\Psi\in\nu_{*}\mathfrak{G}_{\mathbb{C}}(S)$ satisfying (\ref{eq:Real-cocycle-cond})
is compatible with base change in the following manner:

\begin{lemma}
\label{lem:Base-change-correspondence}Let $(f\colon \mathcal{\mathfrak{X}}\to\underline{S},\theta)$
be an $S$-form of $X$ with corresponding automorphism $\Psi$ and
let $h\colon S'\to S$ be a morphism of $\mathbb{R}$-schemes. Then the
base change $(f_{S',h}\colon \mathfrak{X}_{S',h}\to\underline{S'},\theta_{S',h})$
of $(f\colon \mathcal{\mathfrak{X}}\to\underline{S},\theta)$ by $h\colon S'\to S$
corresponds to the $S'_{\mathbb{C}}$-automorphism $\Psi'=\mathrm{id}_{S'_{\mathbb{C}}}\times_{S_{\mathbb{C}}}\Psi$
of $S'_{\mathbb{C}}\times_{\mathbb{C}}X_{\mathbb{C}}\cong S'_{\mathbb{C}}\times_{S_{\mathbb{C}}}(S_{\mathbb{C}}\times_{\mathbb{C}}X_{\mathbb{C}})$. 
\end{lemma}

\begin{proof}
By definition of the base change in Lemma~\ref{lem:Base-change-Def},
we have $\alpha_{\mathfrak{X}_{S',h}}=\alpha_{\underline{S'}}\times_{\underline{S}_{\mathbb{C}}}\alpha_{\mathfrak{X}}$
and $\theta_{S',h}=\mathrm{id}_{\underline{S'_{\mathbb{C}}}}\times_{\underline{S}_{\mathbb{C}}}\theta$.
Thus, 
\[
\alpha_{(\mathfrak{X}_{S',h},\theta_{S',h})}=\tau^{*}(\mathrm{id}_{\underline{S'}_{\mathbb{C}}}\times_{\underline{S}_{\mathbb{C}}}\theta)\circ(\alpha_{\underline{S'}}\times_{\underline{S}_{\mathbb{C}}}\alpha_{\mathfrak{X}})\circ(\mathrm{id}_{\underline{S'}_{\mathbb{C}}}\times_{\underline{S}_{\mathbb{C}}}\theta^{-1})=\alpha_{\underline{S'}}\times_{\underline{S}_{\mathbb{C}}}\alpha_{(\mathfrak{X},\theta)}
\]
and hence, by construction of the above correspondence between forms
and real structures, we have 
\begin{align*}
\Psi'=\tau^{*}\alpha_{\underline{S}'\times\underline{X}}\circ\alpha_{(\mathfrak{X}_{S',h},\theta_{S',h})} & =\tau^{*}\alpha_{\underline{S}'\times\underline{X}}\circ(\alpha_{\underline{S'}}\times_{\underline{S}_{\mathbb{C}}}\alpha_{(\mathfrak{X},\theta)})=\mathrm{id}_{\underline{S}'_{\mathbb{C}}}\times(\tau^{*}\alpha_{\underline{S}\times\underline{X}}\circ\alpha_{(\mathfrak{X},\theta)})=\mathrm{id}_{\underline{S}'_{\mathbb{C}}}\times_{\underline{S}_{\mathbb{C}}}\Psi.
\end{align*}
\end{proof}
Denoting by $\mathfrak{Z}(S)\subset(\nu_{*}\mathfrak{G}_{\mathbb{C}})(S)=\mathfrak{G}_{\mathbb{C}}(S_{\mathbb{C}})$
the subset consisting of all $S_{\mathbb{C}}$-automorphisms $\Psi$
of $S_{\mathbb{C}}\times_{\mathbb{C}}X_{\mathbb{C}}$ that satisfy
(\ref{eq:Real-cocycle-cond}), it follows from Lemma~\ref{lem:Base-change-correspondence}
that the association 
\begin{equation}
S\mapsto\mathfrak{Z}(S)\quad\textrm{and}\quad(h'\colon S'\to S)\mapsto\mathfrak{Z}(h)\coloneqq(\nu_{*}\mathfrak{G}_{\mathbb{C}})(h)|_{\mathfrak{Z}(S)}\colon \mathfrak{Z}(S)\to\mathfrak{Z}(S')\label{eq:Def-Z_X}
\end{equation}
determines a set-valued subfunctor $\mathfrak{Z}$ of $\nu_{*}\mathfrak{G}_{\mathbb{C}}$
with the property that for every $\mathbb{R}$-scheme $S$, the $S$-forms
of $X$ are in one-to-one correspondence with the elements of $\mathfrak{Z}(S)$. 

Letting $\iota\colon \mathfrak{G}\to\mathfrak{G}$ the inverse morphism
of the group functor structure and $\alpha_{\mathfrak{G}}\colon \mathfrak{G}_{\mathbb{C}}\to\tau^{*}\mathfrak{G}_{\mathbb{C}}$
be the canonical real structure on $\mathfrak{G}_{\mathbb{C}}$ (compare with Definition~\ref{def:real group functor} and Example~\ref{exa:AutGrpFunctor}), we
have the following result: 

\begin{proposition}
\label{lem:Z_X-Form} \label{lem:WeilRest-Action-Orbits}The subfunctor $\mathfrak{Z}\subset \nu_*\mathfrak{G}_\C$ is equal to the real form $\mathfrak{Z}_{\mathfrak{G}}=\mathfrak{G}_\C/(\alpha_\mathfrak{G}\circ \iota_\C)$ of the underlying set-valued functor of the group functor $\mathfrak{G}=\underline{\mathrm{Aut}}_{\mathbb{R}}(X)$ defined in $\S$~\ref{sub:realform-funct}. Moreover, 
for every $\R$-scheme $S$, two $S$-forms of $X$ are isomorphic if and only if the corresponding elements of $\mathfrak{Z}(S)=\mathfrak{Z}_\mathfrak{G}(S)$ belong to the same orbit of the action $\gamma(S)\colon\nu_*\mathfrak{G}_\C(S)\times \mathfrak{Z}_\mathfrak{G}(S)\to \mathfrak{Z}_\mathfrak{G}(S)$ defined in \eqref{eq:Weil-act-Z_G}.
\end{proposition}

\begin{proof}
For every $\mathbb{R}$-scheme $S$, $\mathfrak{Z}_{\mathfrak{G}}(S)$ is by definition the subset of $(\nu_{*}\mathfrak{G}_{\mathbb{C}})(S)=\mathfrak{G}_{\mathbb{C}}(S_{\mathbb{C}})$
consisting of real morphisms $h\colon (\underline{S}_{\mathbb{C}},\alpha_{\underline{S}})\to(\mathfrak{G}_{\mathbb{C}},\alpha_{\mathfrak{G}}\circ\iota_{\mathbb{C}})$.
By Yoneda's Lemma, such a morphism $h$ is fully determined by an $S_{\mathbb{C}}$-automorphism $\Psi^{-1}\coloneqq h(S_{\mathbb{C}})(\mathrm{id}_{S_{\mathbb{C}}})\in\mathfrak{G}_{\mathbb{C}}(S_{\mathbb{C}})$ of $S_{\mathbb{C}}\times_{\mathbb{C}}X_{\mathbb{C}}$ such that 
\[
(\alpha_{\underline{S}}\times\mathrm{id}_{\underline{X}_{\mathbb{C}}})\circ\Psi^{-1}\circ(\tau^{*}\alpha_{\underline{S}}\times\mathrm{id}_{\underline{X}_{\mathbb{C}}})=(\mathrm{id}_{\tau^{*}\underline{S}_{\mathbb{C}}}\times\tau^{*}\alpha_{\underline{X}})\circ\tau^{*}\Psi\circ(\mathrm{id}_{\tau^{*}\underline{S}_{\mathbb{C}}}\times\alpha_{\underline{X}}),
\]
equivalently, since $\tau^{*}\alpha_{\underline{S}}\times\tau^{*}\alpha_{\underline{X}}=\tau^{*}\alpha_{\underline{S}\times\underline{X}}$,
such that $\Psi^{-1}=\tau^{*}\alpha_{\underline{S}\times\underline{X}}\circ\tau^{*}\Psi\circ\alpha_{\underline{S}\times\underline{X}}$. The elements $h$ of $\mathfrak{Z}_{\mathfrak{G}}(S)$ thus correspond precisely to the  $S_{\mathbb{C}}$-automorphisms $\Psi$ of $S_{\mathbb{C}}\times_{\mathbb{C}}X_{\mathbb{C}}$ which satisfy the cocycle identity (\ref{eq:Real-cocycle-cond}) hence, by definition, to the elements of $\mathfrak{Z}(S)$. 

The definition of the action $\gamma$ in $\S$~\ref{sub:realform-funct} together with the description of the canonical real structure $\alpha_\mathfrak{G}$ in Example \ref{exa:real-struct-aut} implies that for every $\Phi\in \nu_*\mathfrak{G}_\C(S)$ and $\Psi\in \mathfrak{Z}(S)=\mathfrak{Z}_\mathfrak{G}(S)$ we have $$\gamma(S)(\Phi,\Psi)=(\tau^{*}\alpha_{\underline{S}\times\underline{X}}\circ\tau^{*}\Phi\circ\alpha_{\underline{S}\times\underline{X}})\circ\Psi\circ\Phi^{-1}.$$ 

Comparing with (\ref{eq:Real-conjugacy-cond}), we conclude that the $\nu_*\mathfrak{G}_\C(S)$-orbit of an element $\Psi\in \mathfrak{Z}(S)$ consists precisely of elements $\Psi'\in \mathfrak{Z}(S)$ which determine an $S$-form of $X$ isomorphic to that determined by $\Psi$, which proves the second assertion. 
\end{proof}

Proposition~\ref{lem:Z_X-Form} implies that the moduli functor which associates to every $\mathbb{R}$-scheme
$S$ the set of isomorphism classes of $S$-forms of $X$ is isomorphic
to the quotient functor $\mathfrak{H}_{X}=\mathfrak{Z}/\nu_{*}\mathfrak{G}_{\mathbb{C}}$
of $\mathfrak{Z}$ by the action $\gamma\colon \nu_{*}\mathfrak{G}_{\mathbb{C}}\times_{\underline{\mathbb{R}}}\mathfrak{Z}\to\mathfrak{Z}$
of $\nu_{*}\mathfrak{G}_{\mathbb{C}}$. The precise correspondence
is summarised as follows: 
\begin{corollary}
\label{prop:Families-Bijection}For an $\mathbb{R}$-scheme $X$,
the map which associates to a $S$-form $(f\colon \mathfrak{X}\to\underline{S},\theta)$
of $X$ the $S_{\mathbb{C}}$-automorphism 
\[
\Psi=\tau^{*}(\alpha_{\underline{S}\times\underline{X}}\circ\theta)\circ\alpha_{\mathfrak{X}}\circ\theta^{-1}\in\mathfrak{Z}(S)
\]
of $S_{\mathbb{C}}\times_{\mathbb{C}}X_{\mathbb{C}}$ induces a bijection
between the set of isomorphism classes of $S$-forms of $X$, pointed
by the class of the trivial $S$-form $\mathrm{pr}_{\underline{S}}\colon \underline{S}\times_{\underline{\mathbb{R}}}\underline{X}\to\underline{S}$,
and the set $\mathfrak{H}_{X}(S)$, pointed by the class of $\mathrm{id}_{S_{\mathbb{C}}\times_{\mathbb{C}}X_{\mathbb{C}}}$. 
\end{corollary}

\begin{remark}\label{rem:Z_X-gamma-independant} It is clear from the construction that for an $\R$-scheme $X$, the pair consisting of the functor $\mathfrak{Z}=\mathfrak{Z}_{\mathfrak{G}}$ and its action $\gamma$ of $\nu_*\mathfrak{G}_\C$ depend only on the automorphism group functor $\mathfrak{G}=\underline{\mathrm{Aut}}_\R(X)$ of $X$ and not on $X$ itself. In other words, the functor $\mathfrak{Z}$ together with the action $\gamma$ are common to all $\R$-schemes with automorphism group functors isomorphic to a fixed group-valued contravariant functor $\mathfrak{G}$.  
\end{remark}

\subsubsection{\label{sub:universal-form}The tautological complete relative form $u\colon \mathfrak{U}_{X}\to\mathfrak{Z}$ }

Let $\mu\colon \mathfrak{G}\times_{\underline{\mathbb{R}}}\underline{X}\to\underline{X}$
be the canonical morphism defining the action of $\mathfrak{G}=\underline{\mathrm{Aut}}_{\mathbb{R}}(X)$
on $\underline{X}$ and let $\mu_{\mathbb{C}}=\nu^{*}\mu\colon (\mathfrak{G}\times_{\underline{\mathbb{R}}}\underline{X})_{\mathbb{C}}\to\underline{X}_{\mathbb{C}}$. 
\begin{lemma}
\label{lem:Univ-Family}The morphism of functors 
\[
\beta=(\alpha_{\mathfrak{G}}\circ(\iota_{\mathbb{C}}\circ\mathrm{pr}_{1}))\times(\alpha_{\underline{X}}\circ\mu_{\mathbb{C}})\colon (\mathfrak{G}\times_{\underline{\mathbb{R}}}\underline{X})_{\mathbb{C}}\to\tau^{*}(\mathfrak{G}\times_{\underline{\mathbb{R}}}\underline{X})_{\mathbb{C}}\cong\tau^{*}\mathfrak{G}_{\mathbb{C}}\times_{\tau^{*}\underline{\mathrm{Spec}(\mathbb{C})}}\tau^{*}\underline{X}_{\mathbb{C}}
\]
is a real structure on $(\mathfrak{G}\times_{\underline{\mathbb{R}}}\underline{X})_{\mathbb{C}}$
and the projection $\mathrm{pr}_{1}\colon ((\mathfrak{G}\times_{\underline{\mathbb{R}}}\underline{X})_{\mathbb{C}},\beta)\to(\mathfrak{G}_{\mathbb{C}},\alpha_{\mathfrak{G}}\circ\iota_{\mathbb{C}})$
is a real morphism.
\end{lemma}

\begin{proof}
Put $\hat{\alpha}_{\mathfrak{G}}=\alpha_{\mathfrak{G}}\circ\iota_{\mathbb{C}}$.
Since the composition $\mu_{\mathbb{C}}\circ(\iota_{\mathbb{C}}\circ\mathrm{pr}_{1}\times\mu_{\mathbb{C}})\colon (\mathfrak{G}\times_{\underline{\mathbb{R}}}\underline{X})_{\mathbb{C}}\to\underline{X}_{\mathbb{C}}$
equals the projection $\mathrm{pr}_{2}$, the commutativity of following
diagram of morphisms of functors \[\xymatrix{(\mathfrak{G}\times_{\underline{\mathbb{R}}} \underline{X})_{\mathbb{C}} \ar[d]_{\alpha_{\mathfrak{G}}\times\alpha_{\underline{X}}} \ar[drr]^{\hat{\alpha}_{\mathfrak{G}}\times(\alpha_{\underline{X}}\circ\mu_{\mathbb{C}})}  \ar[rr]^{\iota_{\mathbb{C}}\circ \mathrm{pr}_1\times\mu_{\mathbb{C}}} && (\mathfrak{G}\times_{\underline{\mathbb{R}}} \underline{X})_{\mathbb{C}}\ar[d]^{\alpha_{\mathfrak{G}}\times\alpha_{\underline{X}}} \ar[rr]^{\mu_{\mathbb{C}}} && \underline{X}_{\mathbb{C}} \ar[d]^{\alpha_{\underline{X}}} \\ \tau^{*}(\mathfrak{G}\times_{\underline{\mathbb{R}}} \underline{X})_{\mathbb{C}} \ar[rr]^{\tau^{*}(i_{\mathbb{C}}\circ \mathrm{pr}_1)\times\tau^{*}\mu_{\mathbb{C}}} && \tau^{*}(\mathfrak{G}\times_{\underline{\mathbb{R}}} \underline{X})_{\mathbb{C}} \ar[rr]^{\tau^{*}\mu_{\mathbb{C}}} && \tau^{*}\underline{X}_{\mathbb{C}}}\] implies
that $\tau^{*}\mu_{\mathbb{C}}\circ(\hat{\alpha}_{\mathfrak{G}}\times(\alpha_{\underline{X}}\circ\mu_{\mathbb{C}})))=\alpha_{\underline{X}}\circ\mu_{\mathbb{C}}\circ(\iota_{\mathbb{C}}\circ\mathrm{pr}_{1}\times\mu_{\mathbb{C}})=\alpha_{\underline{X}}\circ\mathrm{pr}_{2}$
and hence that 
\begin{align*}
\tau^{*}\beta\circ\beta & =(\tau^{*}\hat{\alpha}_{\mathfrak{G}}\times(\tau^{*}\alpha_{\underline{X}}\circ\tau^{*}\mu_{\mathbb{C}}))\circ(\hat{\alpha}_{\mathfrak{G}}\times(\alpha_{\underline{X}}\circ\mu_{\mathbb{C}}))=(\tau^{*}\hat{\alpha}_{\mathfrak{G}}\circ\hat{\alpha}_{\mathfrak{G}})\times(\tau^{*}\alpha_{\underline{X}}\circ\alpha_{\underline{X}})=\mathrm{id}_{(\mathfrak{G}\times_{\underline{\mathbb{R}}}\underline{X})_{\mathbb{C}}}.
\end{align*}
The second assertion is clear from the definition of $\beta$. 
\end{proof}
Under the equivalence of categories of Proposition~\ref{lem:Descent-on-functors},
the real functor $((\mathfrak{G}\times_{\underline{\mathbb{R}}}\underline{X})_{\mathbb{C}},\beta)$
corresponds to a pair 
\begin{equation}
\mathfrak{U}_{X}\coloneqq(\mathfrak{G}\times_{\underline{\mathbb{R}}}\underline{X})_{\mathbb{C}}/\beta\in\mathrm{Ob}(\mathcal{F}_{\mathbb{R}})\quad\textrm{and}\quad\Theta_{u}\colon \nu^{*}\mathfrak{U}_{X}\to(\mathfrak{G}\times_{\underline{\mathbb{R}}}\underline{X})_{\mathbb{C}}\in\mathrm{Isom}_{\mathcal{F}_{\mathbb{C}}}(\nu^{*}\mathfrak{U}_{X},(\mathfrak{G}\times_{\underline{\mathbb{R}}}\underline{X})_{\mathbb{C}})\label{eq:Def-U_X}
\end{equation}
and the real morphism $\mathrm{pr}_{1}\colon ((\mathfrak{G}\times_{\underline{\mathbb{R}}}\underline{X})_{\mathbb{C}},\beta)\to(\mathfrak{G}_{\mathbb{C}},\alpha_{\mathfrak{G}}\circ\iota_{\mathbb{C}})$
to a morphism $u\colon \mathfrak{U}_{X}\to\mathfrak{G}_{\mathbb{C}}/(\alpha_{\mathfrak{G}}\circ\iota_{\mathbb{C}})$
such that $u_{X,\mathbb{C}}=\mathrm{pr}_{\mathfrak{G}_{\mathbb{C}}}\circ\mathrm{\Theta}_{u}$.
Through the isomorphism of Proposition~\ref{lem:Z_X-Form}, we can view
these as a morphism $u\colon \mathfrak{U}_{X}\to\mathfrak{Z}$ in $\mathcal{F}_{\mathbb{R}}$
and a $\mathfrak{Z}_{\mathbb{C}}$-isomorphism $\Theta_{u}\colon \nu^{*}\mathfrak{U}_{X}\to(\mathfrak{Z}\times_{\underline{\mathbb{R}}}\underline{X})_{\mathbb{C}}$
in $\mathcal{F}_{\mathbb{C}}$. By construction, the pair $(u\colon \mathfrak{U}_{X}\to\mathfrak{Z},\Theta_{u})$
satisfies the desired property of a complete relative form of $X$,
more precisely: 
\begin{proposition}
For every $\mathbb{R}$-scheme $S$ and every morphism $h\colon \underline{S}\to\mathfrak{Z}$,
the pair $(f\colon \underline{S}\times_{\mathfrak{Z}}\mathfrak{U}_{X}\to\underline{S},\theta_{u,h})$
is the $S$-form of $X$ associated to the $S_{\mathbb{C}}$-automorphism
$\Psi=h(S)(\mathrm{id}_{S})\in\mathfrak{Z}(S)\subset\mathfrak{G}_{\mathbb{C}}(S_{\mathbb{C}})$
of $S_{\mathbb{C}}\times_{\mathbb{C}}X_{\mathbb{C}}$. 
\end{proposition}

\begin{remark} The above description and results have natural variants for $\R$-schemes endowed with additional structures. For instance, the relative forms of a group $\R$-scheme $G$ have a tautological complete relative form $u\colon\mathfrak{U}_G\to \mathfrak{Z}$ consisting of a $\mathfrak{Z}$-group functor $\mathfrak{U}_G$ over the real form $\mathfrak{Z}=\mathfrak{Z}_{\mathfrak{G}}$ of the functor $\mathfrak{G}=(\underline{\mathrm{Aut}_{\mathrm{grp}}})_\R(G)$ of group automorphisms of $G$. In a similar way, relative forms of a pair $(X,Y)$ consisting of an $\R$-scheme $X$ and a closed subscheme $Y$ of $X$ have a tautological complete relative form $u\colon\mathfrak{U}_{(X,Y)}\to \mathfrak{Z}$ over the real form 
$\mathfrak{Z}=\mathfrak{Z}_{\mathfrak{G}}$ of the subfunctor $\mathfrak{G}=\underline{\mathrm{Aut}}_\R(X,Y)$ of $\underline{\mathrm{Aut}}_\R(X)$ consisting of automorphisms of the pair $(X,Y)$.  
\end{remark}

\subsubsection{\label{subsec:Galois-cocycle-real-struct}Comparison with classical
Galois cohomology }

In terms of real structures on $\mathbb{C}$-schemes as in $\S$~\ref{subsec:Classical-real-structures},
an $S$-form of an $\mathbb{R}$-scheme $X$ over an $\mathbb{R}$-scheme
$S$ is given by the choice of an $S_{\mathbb{C}}$-automorphism $\Psi$
of $S_{\mathbb{C}}\times_{\mathbb{C}}X_{\mathbb{C}}$ for which the
composition of $\tau_{*}\Psi$ with the canonical real structure $\sigma_{S\times_{\mathbb{R}}X}=\sigma_{S}\times\sigma_{X}$
on $S_{\mathbb{C}}\times_{\mathbb{C}}X_{\mathbb{C}}$ is a real structure
$\Sigma$ on $S_{\mathbb{C}}\times_{\mathbb{C}}X_{\mathbb{C}}$ such
that the projection $\mathrm{pr}_{1}\colon (S_{\mathbb{C}}\times_{\mathbb{C}}X_{\mathbb{C}},\Sigma)\to(S_{\mathbb{C}},\sigma_{S})$
is a real morphism. By the correspondence of $\S$~\ref{subsec:Functor-Z_X},
automorphisms with this property are exactly those that satisfy the
identity 
\begin{equation}
(\sigma_{S}\times\sigma_{X})\circ\tau_{*}\Psi\circ(\sigma_{S}\times\sigma_{X})^{-1}\circ\Psi=\mathrm{id}_{S_{\mathbb{C}}\times_{\mathbb{C}}X_{\mathbb{C}}}\label{eq:cocycle-relative}
\end{equation}
and two such automorphisms $\Psi$ and $\Psi'$ define $(S_{\mathbb{C}},\sigma_{S})$-isomorphic
real schemes if and only if the identity

\begin{align}
\Psi' & =(\sigma_{S}\times\sigma_{X})\circ\tau_{*}\Phi\circ(\sigma_{S}\times\sigma_{X})^{-1}\circ\Psi\circ\Phi^{-1}\label{eq:coboundary-relative}
\end{align}
holds form some $S_{\mathbb{C}}$-automorphism $\Phi$ of $S_{\mathbb{C}}\times_{\mathbb{C}}X_{\mathbb{C}}$.
Specializing further to $S=\mathrm{Spec}(\mathbb{R})$, we recover
that every real form of $X$ is determined by a real structure $\sigma=\sigma_{X}\circ\tau_{*}\Psi$
for some $\mathbb{C}$-automorphism $\Psi$ of $X_{\mathbb{C}}$ satisfying
the classical cocycle identity $\sigma_{X}\circ\tau_{*}\Psi\circ\sigma_{X}^{-1}\circ\Psi=\mathrm{id}_{X_{\mathbb{C}}}$.
By Proposition\ref{prop:Effective-descent}, such a real structure
$\sigma$ determines a representable form of $X$ if and only if $X_{\mathbb{C}}$
is covered by $\sigma$-stable affine open subschemes. In particular
every real form of a quasi-projective $\mathbb{R}$-scheme is representable
by a quasi-projective $\mathbb{R}$-scheme. 

Letting $\mathrm{Aut}_{\mathbb{C}}(X_{\mathbb{C}})$ be endowed with
the $\Gamma$-group structure~\cite[1.2]{BS64} defined by the action
of the Galois group $\Gamma=\{\pm1\}$ by $\Psi\mapsto\sigma_{X}\circ\tau_{*}\Psi\circ\sigma_{X}^{-1}$,
the above cocycle identity is equivalent to the property that the
map \begin{equation}\label{eq:Galois-cocycle}
c_\sigma\colon\Gamma=\{\pm 1\} \to \mathrm{Aut}_{\mathbb{C}}(X_{\mathbb{C}}), \left\{ \begin{array}{rl} 1\mapsto & \mathrm{id}_{X_{\mathbb{C}}} \\ -1 \mapsto & \tau_*(\sigma_X^{-1}\circ \sigma)=\Psi \end{array} \right. 
\end{equation} is a $1$-cocycle of $\Gamma$ with values in $\mathrm{Aut}_{\mathbb{C}}(X_{\mathbb{C}})$.
We denote the set of such cocycles by $Z^{1}(\Gamma,\mathrm{Aut}_{\mathbb{C}}(X_{\mathbb{C}}))$.
The identity (\ref{eq:Real-conjugacy-cond}) specialized to $\mathrm{Spec}(\mathbb{R})$-forms
is in turn equivalent to the property that two real structures $\sigma=\sigma_{X}\circ\tau_{*}\Psi$
and $\sigma'=\sigma_{X}\circ\tau_{*}\Psi'$ determine isomorphic real
forms of $X$ if and only if $\Psi'=\sigma_{X}\circ\tau_{*}\Phi\circ\sigma_{X}^{-1}\circ\Psi\circ\Phi^{-1}$,
that is, if and only if their associated $1$-cocycles $c_{\sigma}$
and $c_{\sigma'}$ differ by a $1$-coboundary of $\Gamma$ with values
in $\mathrm{Aut}_{\mathbb{C}}(X_{\mathbb{C}})$. 

By construction of the functors $\mathfrak{Z}$ and $\mathfrak{H}_{X}$
in subsection~\ref{subsec:Tauto-form}, the map which associates to
a $1$-cocycle $c_{\sigma}$ the $\mathbb{C}$-automorphism $\Psi=c_{\sigma}(-1)$
of $X_{\mathbb{C}}$ defines a bijection $Z^{1}(\Gamma,\mathrm{Aut}_{\mathbb{C}}(X_{\mathbb{C}}))\to\mathfrak{Z}(\mathrm{Spec}(\mathbb{R}))$
which induces in turn a bijection between the Galois cohomology set
$H^{1}(\Gamma,\mathrm{Aut}_{\mathbb{C}}(X_{\mathbb{C}}))$ pointed
by the class of the constant 1-cocycle with value $\mathrm{id}_{X_{\mathbb{C}}}$
and the set $\mathfrak{H}_{X}(\mathrm{Spec}(\mathbb{R}))$ pointed
by the class of $\mathrm{id}_{X_{\mathbb{C}}}$, the elements of the
latter set being by construction in one-to-one correspondence with
the orbits of the action $\gamma(\mathrm{Spec}(\mathbb{R}))$ of $(\nu_{*}\mathfrak{G}_{\mathbb{C}})(\mathrm{Spec}(\mathbb{R}))=\mathrm{Aut}_{\mathbb{C}}(X_{\mathbb{C}})$
on $\mathfrak{Z}(\mathrm{Spec}(\mathbb{R}))$ defined in 
\eqref{eq:Weil-act-Z_G} in $\S$~\ref{sub:realform-funct}. The isomorphism of Proposition
\ref{prop:Families-Bijection} can thus be thought of as a relative
version of the Borel-Serre correspondence~\cite[Proposition 1.8]{BS64}
between isomorphic classes of real forms of a quasi-projective $\mathbb{R}$-scheme
$X$ and elements of the Galois cohomology set $H^{1}(\Gamma,\mathrm{Aut}_{\mathbb{C}}(X_{\mathbb{C}}))$. 

\subsection{Quasi-projective $\mathbb{R}$-schemes with representable automorphism
group functors}

For an arbitrary $\mathbb{R}$-scheme $X$, the functors $\mathfrak{Z}$
and $\mathfrak{U}_{X}$ are in general not representable. In particular,
due to Proposition~\ref{lem:Z_X-Form}, the representability of $\underline{\mathrm{Aut}}_{\mathbb{C}}(X_{\mathbb{C}})$
is a necessary condition for the representability of $\mathfrak{Z}$.
For a quasi-projective $\mathbb{R}$-scheme whose automorphism group
functor is representable by a group $\mathbb{R}$-scheme locally of
finite type with at most countably many irreducible components\footnote{By~\cite{MO67}, this holds for instance when $X$ is proper, hence
projective in our set-up.}, we have the following result:

\begin{theorem}
\label{thm:Quasiproj-repAut}Let $X$ be a quasi-projective $\mathbb{R}$-scheme
whose automorphism group functor $\mathfrak{G}=\underline{\mathrm{Aut}}_\R(X)$
is representable by a group $\mathbb{R}$-scheme $G$ locally of finite
type with at most countably many irreducible components. Then the
following hold:
\begin{enumerate}[label=(\alph*), leftmargin=*]
	\item The functors $\mathfrak{Z}$, $\mathfrak{U}_{X}$ and $\nu_{*}\mathfrak{G}_{\mathbb{C}}$
	are representable by $\mathbb{R}$-schemes $Z$,
	$U_{X}$ and the group $\mathbb{R}$-scheme $R_{\mathbb{C}/\mathbb{R}}G_{\mathbb{C}}$, respectively,
	which are all locally of finite type with quasi-projective connected
	components.
	\item The action of $(R_{\mathbb{C}/\mathbb{R}}G_{\mathbb{C}})(\mathbb{R})=\mathrm{Aut}_{\mathbb{C}}(X_{\mathbb{C}})$
	on $Z(\mathbb{R})\cong Z^{1}(\Gamma,\mathrm{Aut}_{\mathbb{C}}(X_{\mathbb{C}}))$
	induced by $\gamma\colon \nu_{*}\mathfrak{G}_{\mathbb{C}}\times_{\underline{\mathbb{R}}}\mathfrak{Z}\to\mathfrak{Z}$
	has at most countably many orbits. Moreover, if $G$ is quasi-compact
	then the number of orbits is finite. 
\end{enumerate}
\end{theorem}
\begin{proof}
Every connected component of the group $\mathbb{C}$-scheme $G_{\mathbb{C}}$
representing the functor $\mathfrak{G}_{\mathbb{C}}=\underline{\mathrm{Aut}}_{\mathbb{C}}(X_{\mathbb{C}})$
is a quasi-projective $\mathbb{C}$-scheme isomorphic to the neutral
component $G_{\mathbb{C}}^{0}$ of $G_{\mathbb{C}}$ (\cite[Expos\'e $\mathrm{VI_A}$]{SGA3-I}
and~\cite[\href{https://stacks.math.columbia.edu/tag/0BF6}{Tag 0BF6}]{StackPro}).
Since the orbit of a closed point of $G_{\mathbb{C}}$ under any real
structure $\sigma$ on $G_{\mathbb{C}}$ is contained in a quasi-projective
subscheme of $G_{\mathbb{C}}$ consisting of the union at most two
connected components of $G_{\mathbb{C}}$, it follows from Proposition
\ref{prop:Effective-descent} that the real form $\underline{G}_{\mathbb{C}}/\alpha_{\sigma}$
of $G_{\mathbb{C}}$ is representable by an $\mathbb{R}$-scheme locally
of finite type, with quasi-projective connected components. This applies
in particular to the real form $\mathfrak{Z}\cong\mathfrak{G}_{\mathbb{C}}/(\alpha_{\mathfrak{G}}\circ\iota_{\mathbb{C}})$
of $G_{\mathbb{C}}$, see Proposition~\ref{lem:Z_X-Form}. The same type
of argument applies to conclude the representability of $\nu_{*}\mathfrak{G}_{\mathbb{C}}\cong\nu_{*}\underline{G}_{\mathbb{C}}$
by the Weil restriction $R_{\mathbb{C}/\mathbb{R}}G_{\mathbb{C}}$
of $G_{\mathbb{C}}$, see Example~\ref{exa:WeilRestriction-2}. Finally,
since $X_{\mathbb{C}}$ is assumed to be quasi-projective so that
every finite union of connected components of $G_{\mathbb{C}}\times_{\mathbb{C}}X_{\mathbb{C}}$
is again quasi-projective, we deduce in the same way that the real
form $\mathfrak{U}_{X}$ of $G_{\mathbb{C}}\times_{\mathbb{C}}X_{\mathbb{C}}$,
see Lemma~\ref{lem:Univ-Family}, is representable. The second assertion
is a reformulation of a result due to Labinet~\cite{La22}, see also
\cite{DinetAl22}. 
\end{proof}

Theorem~\ref{thm:Quasiproj-repAut} says in particular that a real
quasi-projective variety $X$ whose automorphism group functor is
representable by an algebraic group $\mathbb{R}$-scheme has finitely
many real forms. Nevertheless, the tautological complete relative form $u_{X}\colon U_{X}\to Z_{X}$
of such a quasi-projective variety displays in general a richer geometry
than what is merely captured by the finiteness of the first Galois
cohomology set $H^{1}(\Gamma,\mathrm{Aut}_{\mathbb{C}}(X_{\mathbb{C}}))$, as illustrated by the following complementary results and examples. 

\begin{corollary} Let $X$ be a smooth quasi-projective $\R$-scheme whose automorphism group functor is representable by an algebraic group $\R$-scheme $G$. Then the following hold: 
\begin{enumerate}[label=(\alph*), leftmargin=*]
  \item The $\R$-schemes $Z$ and $U_X$ representing the functors $\mathfrak{Z}$ and $\mathfrak{U}_X$ are smooth and $u\colon U_X\to Z$ is a smooth morphism. 
  \item The real loci $Z(\R)$ and $U_X(\R)$ are smooth manifolds when endowed with their respective Euclidean topologies and the induced map $u(\R)\colon U_X(\R)\to Z(\R)$ is a surjective submersion\footnote{In particular, if $X$ is in addition proper then $u(\R)\colon U_X(\R)\to Z(\R)$ is locally trivial fibration \cite{Eh50}.} whose fibres are diffeomorphic to the real loci of the corresponding real forms of $X$.
  \item The action $\gamma\colon R_{\C/\R}G_\C\times_\R Z\to Z$ induces a differentiable action $\gamma(\R)\colon R_{\C/\R}G_\C(\R)\times Z(\R)\to Z(\R)$ of the Lie group $R_{\C/\R}G_\C(\R)=G_\C(\C)$ on $Z(\R)$ whose orbits are in one-to-one correspondence with the isomorphism classes of real forms of $X$. 
  \end{enumerate} 
\end{corollary}
\begin{proof} Being an algebraic group $\R$-scheme, $G$ is a smooth $\R$-scheme. Since $X$ is smooth, $G\times_\R X$ is a smooth $\R$-scheme as well. The smoothness of $Z$ and $U_X$ then follows from the fact that smoothness is a local property in the \'etale topology and that these two schemes are by construction real forms of $G$ and $G\times_\R X$ respectively. By its construction in Subsection~$\S$~\ref{sub:universal-form}, the complexification $u_\C\colon U_{X,\C}\to Z_\C$ of $u\colon U_X\to Z$ is isomorphic to the trivial bundle $\mathrm{pr}_1\colon G_\C\times_\C X_\C\to X_\C$, which is a smooth morphism since $X_\C$ is a smooth $\C$-scheme. This implies in turn that  $u\colon U_X\to Z$  is a smooth morphism, which proves the first assertion. The other two follow from Proposition~\ref{lem:Z_X-Form} and standard properties of real loci of smooth $\R$-schemes and morphisms between them. 
\end{proof}

\begin{example} For every connected algebraic group $\R$-scheme $G$, there exists a smooth projective $\R$-scheme $X$ such that the real form $Z_G$ of $G$ representing the functor $\mathfrak{Z}_{G}$ of Subsection~\ref{subsec:familie-basic} appears as a connected component of the base $Z$ of the complete relative form $u\colon U_X\to Z$ of $X$. Indeed, \cite{Br14} provides the existence of a smooth projective $\R$-scheme $X$ having $G$ as the neutral component $\mathrm{Aut}^0_\R(X)$ of its automorphism group $\R$-scheme $\mathrm{Aut}_\R(X)$. Since  $\mathrm{Aut}^0_\R(X)$ is a closed subgroup $\R$-scheme of $\mathrm{Aut}^0_\R(X)$, the real form $Z_G$ of $G$ is a connected component of the real form $Z$ of $\mathrm{Aut}^0_\R(X)$ representing the functor $\mathfrak{Z}$.  
\end{example}

\begin{example}
The affine line $\mathbb{A}_{\mathbb{R}}^{1}$ has no non-trivial
real form. Nevertheless, the scheme $Z$ representing the functor $\mathfrak{Z}$
is the real form $\mathbb{S}^{1}\times_{\mathbb{R}}\mathbb{A}_{\mathbb{R}}^{1}$
of the automorphism group scheme $\mathbb{G}_{m,\mathbb{R}}\rtimes\mathbb{G}_{a,\mathbb{R}}$
of $\mathbb{A}_{\mathbb{R}}^{1}$. The morphism $u\colon U_{X}\to Z$
is the pullback by the projection $\mathrm{pr}_{\mathbb{S}_{1}}\colon \mathbb{S}^{1}\times_{\mathbb{R}}\mathbb{A}_{\mathbb{R}}^{1}\to\mathbb{S}^{1}$
of the non-trivial line bundle $f\colon \mathcal{X}\to\mathbb{S}^{1}$ of Example~\ref{exa:G_m-over-S1}. 
\end{example}

\begin{example}
The punctured affine line $X=\mathbb{A}_{\mathbb{R}}^{1}\setminus\{0\}=\mathbb{G}_{m,\mathbb{R}}$
has two non-trivial real forms up to isomorphism: the norm one torus
$\mathbb{S}^{1}$ and the non-trivial $\mathbb{S}^{1}$-torsor $\hat{\mathbb{S}}^{1}=\mathrm{Spec}(\mathbb{R}[x,y]/(x^{2}+y^{2}+1))$.
The functor $\mathfrak{U}_{X}$ is represented by the disjoint union
$U_{X}$ of two copies $U_{X,\pm}$ of the Weil restriction $R_{\mathbb{C}/\mathbb{R}}\mathbb{G}_{m,\mathbb{C}}$
of $\mathbb{G}_{m,\mathbb{C}}$ and the functor $\mathfrak{Z}$
by the the disjoint union $Z$ of the schemes $Z_{-}=\mathbb{G}_{m,\mathbb{R}}$
and $Z_{+}=\mathbb{S}^{1}$. The morphism 
\[
u=u_{+}\sqcup u_{-}\colon U_{X}=U_{X,+}\sqcup U_{X,-}=R_{\mathbb{C}/\mathbb{R}}\mathbb{G}_{m,\mathbb{C}}\sqcup R_{\mathbb{C}/\mathbb{R}}\mathbb{G}_{m,\mathbb{C}}\to\mathbb{S}^{1}\sqcup\mathbb{G}_{m,\mathbb{R}}=Z_{+}\sqcup Z_{-}=Z
\]
consists of the non-trivial $\mathbb{S}^{1}$-torsor $u_{+}\colon R_{\mathbb{C}/\mathbb{R}}\mathbb{G}_{m,\mathbb{C}}\to R_{\mathbb{C}/\mathbb{R}}\mathbb{G}_{m,\mathbb{C}}/\mathbb{S}^{1}=\mathbb{G}_{m,\mathbb{R}}$
of Example~\ref{exa:Affine-conic-bundle} and of the non-trivial $\mathbb{G}_{m,\mathbb{R}}$-torsor
$R_{\mathbb{C}/\mathbb{R}}\mathbb{G}_{m,\mathbb{C}}\to R_{\mathbb{C}/\mathbb{R}}\mathbb{G}_{m,\mathbb{C}}/\mathbb{G}_{m,\mathbb{R}}=\mathbb{S}^{1}$
of Example~\ref{exa:G_m-over-S1}. 
\end{example}

\begin{example}\label{exa:P1} The projective line $\mathbb{P}^1_\R$ has a unique non-trivial real form, isomorphic to the smooth conic $C=\{x^2+y^2+z^2=0\}$ in $\mathbb{P}^2_\R$. Since $\mathrm{Aut}_{\R}(\mathbb{P}^1_\R)\cong \mathrm{PGL}_{2,\R}$, the base scheme $Z$ of the tautological complete relative form of $\mathbb{P}^1_\R$ is isomorphic to the affine threefold $\mathbb{P}^3_\R\setminus Q_0$ of Example~\ref{exa:PGL2}, where $Q_0$ denotes the zero locus in $\mathbb{P}^3_\R$ of the quadratic form $q_0(x,y,z,t)=x^2+y^2+z^2-t^2$. The complete relative from $u\colon U_{\mathbb{P}^1_\R}\to Z$ identifies in turn with the restriction over $Z=\mathbb{P}^3_\R\setminus Q_0$ of the conic bundle $$\mathrm{pr}_1\colon I=\{xx'+yy'+zz'+tt'=q_0(x',y',z',t')=0 \}\subset \mathbb{P}_{\mathbb{R}}^{3}\times_{\mathbb{R}}\mathrm{Proj}(\mathbb{R}[x',y',z',t'])\to \mathbb{P}^3_\R$$ with discriminant $Q_0$. With the notation of Example~\ref{exa:PGL2}, the fibre of $u\colon U_{\mathbb{P}^1_\R}\to Z$ over an $\R$-rational point of the connected component $Z_-(\R)$ is isomorphic to the non-trivial real form $C$ whereas its fibre over an  $\R$-rational point of the connected component $Z_+(\R)$ is isomorphic to $\mathbb{P}^1_\R$. 
\end{example}

\section{Examples of quasi-projective toric surfaces with moduli of real forms} \label{sec:construction-family}

In this section, we construct examples of non-trivial relative real forms of the quasi-projective surfaces \[X_n=\{u^2+v^2-t^{2n+1}w^2=0\}\subset \mathbb{A}^1_{\mathbb{R}}\times_\R \mathbb{P}^2_{\mathbb{R}}, \quad n\geq 0,\] and their minimal desingularisations $\tilde{X}_n\to X_n$. We begin with a review of the geometry and the automorphism groups of these surfaces and their respective complexifications. We then proceed to the construction of a class of relative real forms of them over affine spaces and to the study of isomorphism classes of their real fibres. 

\subsection{Notations and conventions}

\begin{notation} \label{nota:PGL} 
Given an affine $\C$-scheme $Z=\mathrm{Spec}(A)$, we fix the following identification between the group $\mathrm{PGL}_2(A)$ and the group of $Z$-automorphisms of $Z\times_\C  \mathbb{P}^1_\mathbb{C}=\mathrm{Proj}_{A}(A[w_0,w_1])$: A matrix $M=\left(\begin{smallmatrix} a & b \\ c & d \end{smallmatrix}\right) \in \mathrm{GL}_2(A)$ determines a homogeneous $A$-algebra automorphism of degree $0$ of the graded $A$-algebra $A[w_0,w_1]$, again denoted by $M$, which maps the generators $(w_0,w_1)$ to $(M(w_0,w_1))=(aw_0+bw_1,cw_0+dw_1)$. The latter determines in turn a $Z$-automorphism $\Psi_M$ of $Z\times_\C \mathbb{P}^1_\mathbb{C}$, which depends only on the class of $M$ in $\mathrm{PGL}_2(A)$. 
\end{notation}

\begin{notation} \label{nota:subgroups} 
Given a Laurent polynomial ring $A[t^{\pm 1}]$ over an integral $\mathbb{C}$-algebra $A$, every element of $\mathrm{PGL}_2(A[t^{\pm 1}])$ has a representative given by a matrix 
\begin{equation} \label{eq:representative}
M=\left(\begin{smallmatrix} a & b \\ c & d \end{smallmatrix}\right)\in \mathrm{GL}_2(A[t^{\pm 1}])\cap (\mathrm{Mat}_{2,2}(A[t])\setminus \mathrm{Mat}_{2,2}(tA[t])).
\end{equation}
We denote by $J_{A,t}$ the class of the matrix $\left(\begin{smallmatrix}
	0 & t \\
	1 & 0
\end{smallmatrix}\right)$ in  $\PGL_2(A[t^{\pm 1}])$. We denote by $G_{A,n}^0$ the subgroup of $\PGL_2(A[t^{\pm 1}])$ consisting of classes of matrices of the form \[ \left(\begin{matrix} a & b \\ c & d\end{matrix}\right)=\left(\begin{matrix} a & t^{n+1}\beta \\ t^{n}\gamma & d\end{matrix}\right) \in \mathrm{GL}_2(A[t^{\pm 1}])\cap (\mathrm{Mat}_{2,2}(A[t])\setminus \mathrm{Mat}_{2,2}(tA[t])) \]  
where $a(0)d(0)$ belongs to the group $A^*$ of invertible elements of $A$. We let $G_{A,n}=G_{A,n}^0\rtimes \langle J_{A,t}\rangle\cong G_{A,n}^0\rtimes \mathbb{Z}/2\mathbb{Z}$ be the subgroup of $\PGL_2(A[t^{\pm 1}])$ generated by $J_{A,t}$ and $G_{A,n}^0$.
\end{notation}

\begin{notation}\label{not:basic} Using the complex coordinate change $x=u+iv$, $y=u-iv$ and $z=w$ on $\mathbb{A}^1_{\mathbb{C}}\times_\C \mathbb{P}^2_{\mathbb{C}}$, we henceforth identify for every $n\geq 0$ the complexification $X_{n,\C}$ of the quasi-projective surface 
\[X_n=\{u^2+v^2-t^{2n+1}w^2=0\}\subset \mathbb{A}^1_{\mathbb{R}}\times_\R \mathbb{P}^2_{\mathbb{R}}\] endowed with its canonical real structure $\sigma_{X_n}$ with the hypersurface 
\begin{equation}\label{eq:Yn} Y_n=\{xy-t^{2n+1}z^2=0\} \subset \mathbb{A}^1_\mathbb{C}\times_\C \mathbb{P}^2_{\mathbb{C}}=\mathrm{Proj}_{\mathbb{C}[t]}(\mathbb{C}[t][x,y,z])\end{equation} 
endowed with the real structure $\sigma$ defined as the restriction of the composition of the involution  
\begin{equation}\label{eq:iota} \iota_{x,y}\colon(t,[x:y:z])\mapsto (t,[y:x:z])
\end{equation} 
with the canonical real structure $\sigma_{\mathbb{A}^1_\mathbb{R}}\times \sigma_{\mathbb{P}^2_\mathbb{R}}$ on $\mathbb{A}^1_{\mathbb{C}}\times_\C \mathbb{P}^2_{\mathbb{C}}$.

We put $B=\mathrm{Spec}(\mathbb{R}[t])$ and denote by $\pi_\C\colon (Y_n,\sigma)\to (B_\mathbb{C},\sigma_B)$ the real projective morphism induced by the first projection $\pi\colon Y_n\to \mathbb{A}^1_\R$. We let $B^{\star}=B\setminus \{0\}=\mathrm{Spec}(\mathbb{R}[t^{\pm 1}])$ and $Y_n^\star=Y_n\setminus \pi_\C^{-1}(0)$.
\end{notation}

\subsection{Basic geometric properties} \label{subsec:geom}

\subsubsection{Conic bundle structure} \label{subsub:cb} 
The morphism $\pi_\C\colon(Y_n,\sigma)\to (B_\mathbb{C},\sigma_B)$ is a real conic bundle whose restriction over $B^\star$ is a $B^\star$-form of $\mathbb{P}^1_{\R}$. Namely, let $B_\mathbb{C}^\star\times_\C \mathbb{P}^1_\mathbb{C}$ be endowed with the real structure $\epsilon$ defined as the composition of the involution $(t,[w_0:w_1])\mapsto (t,[tw_1:w_0])$ with the canonical real structure $\sigma_{B^\star}\times \sigma_{\mathbb{P}^1_\mathbb{R}}$. Then the real closed immersion 
 \begin{equation}
 \label{eq:bundle-triv}
 \begin{array}{rcl}
 c\colon (B_\mathbb{C}^\star\times_\C \mathbb{P}^1_\mathbb{C},\epsilon) & \longrightarrow &
 (B_\mathbb{C}^\star\times_\C \mathbb{P}^2_\mathbb{C}, (\sigma_{B^\star}\times \sigma_{\mathbb{P}^2_\mathbb{R}})\circ \tau_*\iota_{x,y} ) \\
  (t,[w_0:w_1]) & \mapsto & (t,[w_0^2:tw_1^2:t^{-n}w_0w_1]),
 \end{array}     
 \end{equation}
of schemes over $(B_\mathbb{C}^\star,\sigma_{B^\star})$ induces a real isomorphism $\beta\colon(B_\mathbb{C}^\star\times_\C \mathbb{P}^1_\mathbb{C},\epsilon)\to (Y_n^\star,\sigma|_{Y_n^\star})$. On the other hand, the fibre $\pi_\C^{-1}(0)$ consists of two irreducible components $L_x=\{t=x=0\}$ and $L_y=\{t=y=0\}$, both isomorphic to $\mathbb{P}^1_\mathbb{C}$, which are exchanged by the involution $\iota_{x,y}$, hence by the real structure $\sigma$.

\subsubsection{Toric structure}\label{subsub:toric} 

The Weil restriction torus $(\mathbb{T},\sigma_\mathbb{T})\coloneqq((R_{\mathbb{C}/\mathbb{R}}\mathbb{G}_{m,\mathbb{C}})_\C,\sigma_{R_{\mathbb{C}/\mathbb{R}}\mathbb{G}_{m,\mathbb{C}}})\cong (\mathbb{G}_{m,\mathbb{C}}^2,\sigma_R)$, where $\sigma_R$ is the composition of the involution exchanging the two factors of $\mathbb{G}_{m,\mathbb{C}}$ with the canonical real structure $\sigma_{\mathbb{G}_{m,\mathbb{R}}^2}$ (see Example~\ref{exa:WeilRestriction-2}), acts on $(Y_n,\sigma)$ by  
\begin{equation}
\label{eq:R-action}
\xi\colon(\mathbb{T},\sigma_\mathbb{T})\times (Y_n,\sigma)\to (Y_n,\sigma), \quad 
\left((\lambda,\mu),(t,[x:y:z])\right)\mapsto (\lambda\mu t,[\lambda x:\mu y:(\lambda\mu)^{-n}z]),
\end{equation}
with open orbit  $Y_n\setminus\{tz=0\}$. Letting $M\cong \mathbb{Z}^2$ be the character lattice of $R_\C$ endowed with the action of the Galois group $\Gamma=\{\pm 1\}$ exchanging the two factors of $\mathbb{Z}^2$ induced by the real structure $\sigma_R$, $Y_n$ is described by the $\Gamma$-stable fan in $M^\vee\otimes_{\mathbb{Z}} \mathbb{Q}$ consisting of three maximal cones: a $\Gamma$-stable one $C_0$ corresponding to the $\mathbb{T}$-fixed point $(0,[0:0:1])$ of $Y_n$ and a pair of regular cones $C_x$ and $C_y$ exchanged by the $\Gamma$-action, corresponding respectively to the $\mathbb{T}$-fixed points $(0,[1:0:0])$ and $(0,[1:0:0])$  of $Y_n$, see Figure~\ref{fig:toric-pic}. 
 
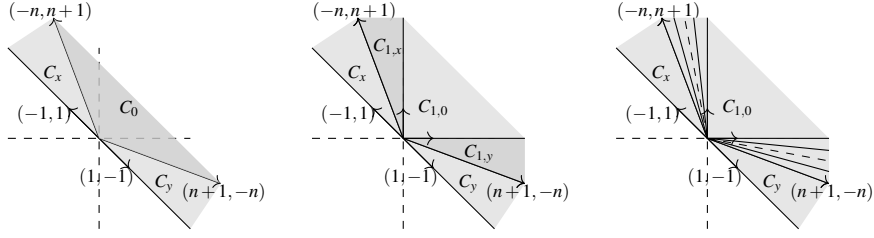
\begin{figure}[htb]
\begin{center}
\begin{tikzpicture}[x=0.2cm,y=0.2cm]

\foreach \a in {0,20,40}
{
\fill[color=gray!20] (\a+0,0)--(\a-6,6)--(\a-3,8);    
\fill[color=gray!20] (\a+0,0)--(\a+6,-6)--(\a+8,-3);  
\draw [->] (\a+0,0) -- (\a+8,-3);
\draw [->] (\a+0,0.) -- (\a-3.,8);
\draw  (\a+0,0) -- (\a-6,6);
\draw (\a+0,0) -- (\a +6,-6);  
\draw [thin, dashed] (\a-6,0)--(\a+6,0);
\draw [thin, dashed] (\a+0,-6)--(\a+0,6);
\draw [->] (\a+0,0) -- (\a-2,2);
\draw [->] (\a+0,0) -- (\a+2,-2); 
\begin{scriptsize}
\draw[color=black] (\a+ 8.2,-3.6) node {$(n+1,-n)$};
\draw[color=black] (\a-3.2,8.4) node {$(-n,n+1)$};
\draw[color=black] (\a+0.5,-2.6) node {$(1,-1)$};
\draw[color=black] (\a-3.6,1.6) node {$(-1,1)$};
\draw[color=black] (\a+4.3,-3) node {$C_y$};
\draw[color=black] (\a-3,4.3) node {$C_x$};
\end{scriptsize}
} 

\fill[color=gray!40, opacity=0.8] (0,0)--(8,-3)--(-3,8);
\begin{scriptsize}
\draw[color=black] (2,2) node {$C_0$};
\end{scriptsize}

\fill[color=gray!20] (20,0)--(28,0)--(20,8);
\fill[color=gray!40, opacity=0.8] (20,0)--(20,8)--(17,8);
\draw[->] (20,0)--(17,8);
\fill[color=gray!40,opacity=0.8] (20,0)--(28,0)--(28,-3);
\draw[->] (20,0)--(28,-3);
\draw (20,0)--(20,8);  
\draw [->] (20,0)--(20,2);  
\draw (20,0)--(28,0);  
\draw [->] (20,0)--(22,0); 
\begin{scriptsize} 
\draw[color=black] (22,2) node {$C_{1,0}$};
\draw[color=black] (19,6) node {$C_{1,x}$};
\draw[color=black] (25,-1) node {$C_{1,y}$};
\end{scriptsize}

\fill[color=gray!20] (40,0)--(48,0)--(40,8);
\fill[color=gray!20] (40,0)--(40,8)--(37,8);
\draw[->] (40,0)--(37,8);
\fill[color=gray!20] (40,0)--(48,0)--(48,-3);
\draw[->] (40,0)--(48,-3);
\draw (40,0)--(40,8);  
\draw [->] (40,0)--(40,2);  
\draw (40,0)--(48,0);  
\draw [->] (40,0)--(42,0); 

\draw [thin] (40,0)--(48,-0.8);
\draw [very thin,dashed] (40,0)--(48,-1.5); 
\draw [thin] (40,0)--(48,-2.2); 

\draw [thin] (40,0)--(39.1,8);
\draw [very thin,dashed] (40,0)--(38.5,8); 
\draw [thin] (40,0)--(37.8,8); 

\begin{scriptsize} 
\draw[color=black] (42,2) node {$C_{1,0}$};
\end{scriptsize}

\end{tikzpicture}
\end{center}
\caption{The fan of the real toric surface $(Y_n,\sigma)$ and its successive subdivisions corresponding to the real toric surfaces $(Y_{n,1},\sigma_1)$ and the minimal desingularisation $(\tilde Y_n,\tilde{\sigma})$ of $(Y_n,\sigma)$.}
\label{fig:toric-pic}  
\end{figure}

\subsubsection{Minimal real desingularisation}\label{subsub:desing} 
The real surface $(Y_n,\sigma)$ is smooth if $n=0$ and for $n\geq 1$ it has a unique singular rational double point $(0,[0:0:1])=L_x\cap L_y$ of type $A_{2n}^+$ in the nomenclature of \cite[Definition 9.3]{Kol99}. The minimal desingularisation $d_\C\colon(\tilde{Y}_n,\tilde{\sigma})\to (Y_n,\sigma)$ is a real toric morphism which can be described as follows (see Figure~\ref{fig:toric-pic}). One first gets a real toric morphism $d_1\colon(Y_{n,1},\sigma_1)\to (Y_n,\sigma)$ where $(Y_{n,1},\sigma_1)$ is the real toric surface whose fan is obtained from that of $Y_n$ by inserting the rays generated by $(0,1)$ and $(1,0)$ to get a new $\Gamma$-stable regular cone $C_{1,0}$ corresponding to a smooth real point of $Y_{n,1}$ and a pair of singular cones $C_{1,x}$ and $C_{1,y}$ 
exchanged by the $\Gamma$-action corresponding to a pair of rational double points of type $A_{n-1}$ of $Y_{n,1}$ exchanged by the real structure $\sigma_1$. Then one gets a second real toric morphism $d_2\colon(\tilde{Y}_n,\tilde{\sigma})\to (Y_{n,1},\sigma_{1})$ by taking the minimal desingularisation of these two points. The total transform by $d_\C$ of $\pi_\C^{-1}(0)=L_x\cup L_y$ is chain of smooth rational curves 
\[\mbox{\small $(L_x',-1)-(E_{n-1,x},-2)-\cdots - (E_{1,x},-2)-(E_{0,x},-2)-(E_{0,y},-2)-(E_{1,y},-2)-\cdots -(E_{n-1,y},-2)-(L_y',-1)$}\]
where $L_x'$ and $L_y'$ are the proper transforms of $L_x$ and $L_y$, the divisors $E_{0,x}$ and $E_{0,y}$ are the proper transforms of the exceptional divisors of $d_1$ and the $E_{i,x}$ and $E_{i,y}$, $i=1,\ldots, n-1$, are the exceptional divisors of $d_2$, the numbers indicated in the parenthesis being the self-intersection numbers of these curves in $\tilde{Y}_n$. 

The real structure $\tilde{\sigma}$ on $\tilde{Y}_n$ exchanges $L_x'$ and $L_y'$ and, for every $i=0,\ldots, n$, $E_{i,x}$ with $E_{i,y}$. The point $E_{0,x}\cap E_{0,y}$ is the unique real point of the fibre of the real morphism $\tilde{\pi}_\C\coloneqq\pi_\C\circ d_\C\colon(\tilde{Y}_n,\tilde{\sigma})\to (B_\mathbb{C},\sigma_B)$ over the point $0$. We put $\tilde{Y}_n^\star=\tilde{Y}_n\setminus\pi_\C^{-1}(0)$ and denote by $d\colon \tilde{X}_n=\tilde{Y}_n/\tilde{\sigma}\to X_n=\tilde{Y}_n/\sigma$ the corresponding desingularisation of $X_n$.

\subsection{Automorphism groups}\label{subsec:autos}

\subsubsection{Conic bundle automorphisms}\label{subsub:cb-aut}
 Let $\Aut_{B_\mathbb{C}}(Y_n)$ be the group of $B_{\mathbb{C}}$-automorphisms of $Y_n$ and let $\Aut_{B_\mathbb{C}}^0(Y_n)$ be its subgroup consisting of automorphisms which leave globally invariant each of the two irreducible components of $\pi_\C^{-1}(0)$. Since every element of $\Aut_{B_\mathbb{C}}(Y_n)$ either preserves or exchanges these two components, the group $\Aut_{B_\mathbb{C}}^0(Y_n)$ is a subgroup of index $2$ of $\Aut_{B_\mathbb{C}}(Y_n)$ which fits in a short exact sequence
\begin{equation}
	1 \rightarrow \Aut_{B_\mathbb{C}}^0(Y_n) \rightarrow \Aut_{B_\mathbb{C}}(Y_n) \rightarrow \Z/2\Z \rightarrow 0,  \label{eq: SES}
\end{equation}  
with a splitting given by the involution $\iota_{x,y}$. The composition of the isomorphism $\beta\colon B_\mathbb{C}^\star\times_\C \mathbb{P}^1_\mathbb{C}\to Y_n^\star$ (see $\S$~\ref{subsub:cb}) with the open immersion  $j\colon Y_n^\star\hookrightarrow Y_n$ induces an injective group homomorphism \[(j\circ \beta)^*\colon\Aut_{B_\mathbb{C}}(Y_n)\hookrightarrow \Aut_{B_{\mathbb{C}}^\star}(B_\mathbb{C}^\star \times_\C \mathbb{P}^1_\mathbb{C})\cong \mathrm{PGL}_2(\mathbb{C}[t^{\pm 1}]).\]
 
\begin{lemma} \label{lemma: explicit forms of matrices in G_n^0} 
With the notation above, the following hold:
\begin{enumerate}[label=(\alph*), leftmargin=*]
	\item $(j\circ \beta)^*\iota_{x,y}=J_{\C,t}$, \label{item: J(t)}
	\item $(j\circ \beta)^*\Aut_{B_\mathbb{C}}^0(Y_n)=G_{\C,n}^0$. \label{item:G_n^0}
\end{enumerate}
\end{lemma}
\begin{proof} The closed immersion $c\colon B_\mathbb{C}^\star\times_\C \mathbb{P}^1_\mathbb{C}\to B_\mathbb{C}^\star\times_\C \mathbb{P}^2_\mathbb{C}$ in~\eqref{eq:bundle-triv} induces an injective group homomorphism
\begin{equation} \label{eq:PGL2-embeds-PGL3}
c^*\colon\mathrm{PGL}_2(\mathbb{C}[t^{\pm 1}])\cong  \mathrm{Aut}_{B_\mathbb{C}^\star}(B_\mathbb{C}^\star\times_\C \mathbb{P}^1_\mathbb{C}) \hookrightarrow \mathrm{Aut}_{B_\mathbb{C}^\star}(B_\mathbb{C}^\star\times_\C \mathbb{P}^2_\mathbb{C})\cong \mathrm{PGL_3}(\mathbb{C}[t^{\pm 1}])
\end{equation} 
which maps an element of $\mathrm{PGL}_2(\mathbb{C}[t^{\pm 1}])$ represented by a matrix $M=\left(\begin{smallmatrix} a&b\\c&d \end{smallmatrix}\right)\in \mathrm{GL}_2(\mathbb{C}[t^{\pm 1}])$ to the element of $\mathrm{PGL}_3(\mathbb{C}[t^{\pm 1}])$ represented by the matrix 
\begin{align*}
		N(M)=\left(\begin{array}{ccc}
			a^2&b^2t^{-1}&2abt^n \\
			c^2t&d^2&2cdt^{n+1} \\
			act^{-n}&bdt^{-(n+1)}&ad+bc
		\end{array}\right) \in \mathrm{GL}_3(\C[t^{\pm 1}]).
	\end{align*}
Assertion~\ref{item: J(t)} follows immediately from this description. To prove Assertion~\ref{item:G_n^0}, it suffices to determine which matrices $N(M)$ represent $B_\mathbb{C}^\star$-automorphisms $g\colon(t,[x:y:z])\mapsto (t,[N(x,y,z)])$ of $Y_n^\star$ which extend to $B_\mathbb{C}$-automorphisms of $Y_n$ leaving globally invariant each of the two irreducible components $L_x$, $L_y$ of $\pi_\C^{-1}(0)$. 
The map associated to $N(M)$ for a matrix $M$ representing an element of $G_{\C,n}^0$ as in Notation~\ref{nota:subgroups} restricts for $t=0$ to the automorphism 
\begin{equation}\label{eq:Aut-restriction-singfibre}
(0, [x:y:z])\mapsto (0,[a(0)^2x:d(0)^{2}y:a(0)\gamma(0)x+d(0)\beta(0)y+a(0)d(0) z]
\end{equation}
of $\pi_\C^{-1}(0)$. This implies that $G_{\C,n}^0\subseteq (j\circ \beta)^*\Aut_{B_\mathbb{C}}^0(Y_n)$. 
  To show the converse inclusion, let  $g \in \Aut_{B_{\mathbb{C}}}^0(Y_n)$ and let $M=\left(\begin{smallmatrix} a & b \\ c & d \end{smallmatrix}\right)$ be a matrix as in~\eqref{eq:representative} representing $g|_{Y_n^\star}$. Since $L_y$ is $g$-stable and $g$ fixes the point $(0,[0:0:1])$, it follows that $g$ maps the point $(0,[1:0:0])$ to some point of $L_y$ with non-zero first coordinate. From this information, we gather that $a(0)^2 \in \C^*$ and that $c=t^n\gamma$ for some $\gamma \in \C[t]$, as $(act^{-n})|_{t=0} \in \C^*$. A similar analysis using $L_x$ instead implies that $d(0)^2\in \C^*$ and that $b=t^{n+1}\beta$ for some $\beta \in \C[t]$. Thus, $M$ represents an element of $G_{\C,n}^0$. 
\end{proof}
 
\subsubsection{Complex and real automorphisms}\label{subsub:aut} \label{subsub:-realaut} 
Since $\pi_\C\colon Y_n\to B_\mathbb{C}$ is proper with $\pi_\C^{-1}(0)$ as unique singular fibre and  $B_{\mathbb{C}}$ is affine, every $\mathbb{C}$-automorphism $\Psi$ of $Y_n$ maps closed fibres of $\pi_\C$ onto closed fibres of $\pi_\C$ and stabilizes $\pi_\C^{-1}(0)$. Letting $s\colon B_\mathbb{C}\to Y_n$ be the section of $\pi_\C$ defined by $t\mapsto (t,[t^{n+2}:t^{n+1}:t])$, it follows that $\psi\coloneqq\pi_\C \circ \Psi \circ s$ is a $\mathbb{C}$-automorphism of $B_\mathbb{C}$ such that $\pi_\C\circ \Psi=\psi\circ \pi_\C$ and which fixes the origin $0$. This yields a group homomorphism 
\[\rho\colon\mathrm{Aut}_{\mathbb{C}}(Y_n)\to \Aut_\C(B_\C,\{0\})\cong (\C^*,\times), \quad \Psi\mapsto \pi_\C\circ \Psi \circ s,\]
with a section given by mapping $(\C^*,\times)$ onto the subtorus $T_1=(\C^*,\times) \times \{1\}$ of $\mathbb{T}(\C)$ in~\eqref{eq:R-action}. The kernel of $\rho$ is equal to $\Aut_{B_{\mathbb{C}}}(Y_n)$ and, in sum, we obtain the following description:

\begin{proposition} \label{prop: Aut(X_n)}
With Notation~\ref{nota:subgroups}, the group $\Aut_{\mathbb{C}}(Y_n)$ is isomorphic to \[\Aut_{B_{\mathbb{C}}}(Y_n)\rtimes \Aut_\C(B_\C,\{0\}) \cong (\Aut_{B_\mathbb{C}}^0(Y_n)\rtimes \langle\iota_{x,y}\rangle)\rtimes T_1 \cong (G_{\C,n}^0 \rtimes \langle J_{\C,t}\rangle)\rtimes (\C^*,\times) \cong  G_{\C,n} \rtimes (\C^*,\times)\]
where the group law on $G_{\C,n} \rtimes (\C^*,\times)$ is induced by the composition law $(M'(t),\nu')\cdot (M(t),\nu)=(M'(\nu t)M(t),\nu'\nu)$ on $\mathrm{GL}_2(\mathbb{C}[t^{\pm 1}])\times \C^*$. Moreover, the group homomorphism $d_\C^*\colon\Aut_{\mathbb{C}}(Y_n) \to \Aut_{\mathbb{C}}(\tilde{Y}_n)$ induced by
the minimal desingularisation $d_\C\colon\tilde{Y}_n\to Y_n$ is an isomorphism. 
\end{proposition}

\begin{proof} The first assertion is immediate from the above explanation and Lemma~\ref{lemma: explicit forms of matrices in G_n^0}. For the second one, recall that the existence and injectivity of $d_\C^*$ follows from general properties of the Zariski-Lipman desingularisation of surfaces (see e.g. \cite[$\S$ 2]{Lip69}). Let $\tilde{\Psi}$ be a $\mathbb{C}$-automorphism of $\tilde{Y}_n$ and let $\Psi\coloneqq d_\C\circ \tilde{\Psi}\circ d_\C^{-1}$ so that $d_\C^*\Psi=\tilde{\Psi}$ as birational maps. Since $\tilde{\pi}_\C=\pi_\C\circ d_\C$ is proper and $B_\mathbb{C}$ is affine, $\tilde{\Psi}$ maps fibres of $\tilde{\pi}_\C$ onto fibres of $\tilde{\pi}_\C$ and preserves the unique singular fibre $\tilde{\pi}^{-1}(0)$ of $\tilde{\pi}_\C$. The induced automorphism of $\tilde{\pi}_\C^{-1}(0)$ preserves adjacency of components and their self-intersections. With the notation of $\S$~\ref{subsub:desing}, it follows that $\tilde{\Psi}$ maps $L_x'\cup L_y'$ isomorphically onto itself and hence that $\Psi$ and its birational inverse do not contract any curve. Since $Y_n$ is normal, Zariski main theorem~\cite[Proposition 4.4.1]{EGAIII-I} 
implies that $\Psi$ is a biregular automorphism of $Y_n$.
\end{proof}

Under the equivalence of categories of Proposition~\ref{lem:Descent-on-functors}, the group $\mathrm{Aut}_{\mathbb{R}}(X_n)$ of $\R$-automorphisms of $X_n$ equals the subgroup $\Aut_{\mathbb{C}}(Y_n)^\sigma$ of $\Aut_\mathbb{C}(Y_n)$ consisting of $\mathbb{C}$-automorphisms $\Psi$ of $Y_n$ which are invariant under the action of the Galois group $\Gamma$ by the conjugation $\Psi\mapsto \sigma \circ \tau_*\Psi \circ \sigma^{-1}$.The involution $\iota_{x,y}$ of~\eqref{eq:iota} has this property. On the other hand, the faithful real action~\eqref{eq:R-action} of the real torus $(\mathbb{T},\sigma_\mathbb{T})$ corresponds under the isomorphism of Proposition~\ref{prop: Aut(X_n)} to an injective homomorphism 
\[ \mathbb{T}(\C)=(\C^*,\times)^2\to  \Aut_{\mathbb{C}}(Y_n)\cong   G_{\C,n} \rtimes (\C^*,\times), \quad (\lambda,\mu)\mapsto \mathrm{Class \; of\;} (\left(\begin{smallmatrix} \lambda & 0 \\ 0 & 1 \end{smallmatrix}\right),\lambda\mu)\] which maps $\mathbb{T}(\C)^{\sigma_\mathbb{T}}=\{(\lambda,\overline{\lambda})\in \mathbb{T}(\mathbb{C})\}$ to the subgroup of  $\Aut_{\mathbb{C}}(Y_n)^\sigma$ consisting of classes of pairs $(\left(\begin{smallmatrix} \lambda & 0 \\ 0 & 1 \end{smallmatrix}\right),\lambda\overline{\lambda})$.

\begin{corollary} \label{cor:Real-Aut}
	The group $\Aut_{\mathbb{R}}(X_n)$ is isomorphic to \[\mathbb{T}(\mathbb{C})^{\sigma_\mathbb{T}} \rtimes \langle\iota_{x,y}\rangle \cong (R_{\mathbb{C}/\mathbb{R}}\mathbb{G}_{m,\mathbb{C}})(\mathbb{R}) \rtimes \mathbb{Z}/2\mathbb{Z} \cong (\mathbb{C}^*,\times)\rtimes \mathbb{Z}/2\mathbb{Z}.\]
Moreover, the group homomorphism $d^*\colon\Aut_{\mathbb{R}}(X_n) \to \Aut_{\mathbb{R}}(\tilde{X}_n)$ induced by the minimal desingularisation morphism $d_\C\colon (\tilde{Y}_n,\tilde{\sigma})\to (Y_n,\sigma)$ is an isomorphism. 
\end{corollary}

\begin{proof}
For a Laurent polynomial $p=\sum p_i t^i \in \mathbb{C}[t^{\pm 1}]$, we put $\overline{p}=\sum \overline{p}_it^i \in \mathbb{C}[t^{\pm 1}]$. Under the isomorphism of Proposition~\ref{prop: Aut(X_n)}, the $\Gamma$-action on $\Aut_\mathbb{C}(Y_n)$ 
corresponds to the restriction to $G_{\C,n} \rtimes (\C^*,\times)$ of the $\Gamma$-action on $\mathrm{PGL}_2(\mathbb{C}[t^{\pm 1}])\rtimes (\C^*,\times)$ that associates to an element represented by a pair $(M=\left(\begin{smallmatrix} a & b \\ c & d \end{smallmatrix}\right),\nu)$, where $M$ is as in~\eqref{eq:representative}, the class in $\mathrm{PGL}_2(\mathbb{C}[t^{\pm 1}])\rtimes (\C^*,\times)$ of the pair 
\begin{equation}\label{eq:Gamma-group-struct}
 (M'\coloneqq \left(\begin{matrix} \overline{\nu} \overline{d} & \overline{\nu} t \overline{c} \\ t^{-1} \overline{b} & \overline{a} \end{matrix}\right),\overline{\nu})\in \mathrm{GL}_2(\mathbb{C}[t^{\pm 1}])\times \mathbb{C}^*. 
\end{equation}
Thus, a pair $(M,\nu)$ represents a $\Gamma$-invariant element of $\mathrm{PGL}_2(\mathbb{C}[t^{\pm 1}])\rtimes (\C^*,\times)$ if and only if $\nu=\overline{\nu}$ and the matrices $M$ and $M'$ differ by the multiplication with an element $\xi\in \mathbb{C}[t^{\pm 1}]^*$ 
such that $\xi\overline{\xi}=\nu$. It follows that $\xi\in \mathbb{C}^*$ and that
\[M=\left(\begin{matrix} a & b \\ \xi^{-1}t^{-1} \overline{b} & \xi^{-1}\overline{a} \end{matrix}\right)\in \mathrm{GL}_2(\mathbb{C}[t^{\pm 1}]) \cap (\mathrm{Mat}_{2,2}(\mathbb{C}[t])\setminus \mathrm{Mat}_{2,2}(t\mathbb{C}[t])).\]
From the condition that $\det(M)=\xi^{-1}(a\overline{a}-tb\overline{b})\in \mathbb{C}[t^{\pm 1}]^*$ and the observation that the polynomials $a\overline{a}$ and $b\overline{b}$ have even vanishing order at $0$, we infer that either $a\in \mathbb{C}^*$ and $b=0$ or $a=0$ and $b=t\beta$ for some $\beta\in \mathbb{C}^*$. 
Thus, $(\mathrm{PGL}_2(\mathbb{C}[t^{\pm 1}])\rtimes (\C^*,\times))^{\Gamma}$ is the subgroup generated by the element $(J_{\C, t},1)$ corresponding to $\iota_{x,y}$  (see  Lemma~\ref{lemma: explicit forms of matrices in G_n^0}~\ref{item: J(t)}) and the classes of pairs $(\left(\begin{smallmatrix} \lambda & 0 \\ 0 & 1 \end{smallmatrix}\right),\lambda\overline{\lambda})$, where $\lambda \in \mathbb{C}^*$.
Noticing that all these classes belong to $G_{\C,n}\rtimes (\C^*,\times)$ gives 
\begin{equation}\label{eq:Aut-Inv-equal}
(\mathrm{PGL}_2(\mathbb{C}[t^{\pm 1}])\rtimes (\C^*,\times))^{\Gamma}=(G_{\C,n}\rtimes (\C^*,\times))^\Gamma\cong \Aut_\mathbb{C}(Y_n)^\sigma.
\end{equation} 
The second assertion follows from the construction of the homomorphism $d_\C^*$ and Proposition~\ref{prop: Aut(X_n)}.
\end{proof}

\subsection{A construction of relative forms}\label{subsub:Mp} \label{lem:real-struct-Weil}\label{prop:family-Weil}\label{sec:families}

In this subsection, we construct relative forms of the surfaces $X_n$ and $\tilde{X}_n$ over the affine space $\mathbb{A}^{m+1}_{\R}$ and determine isomorphism classes of their real fibres. We then proceed to the proofs of Theorem~\ref{MainThm} and Theorem~\ref{MainThm-2}.
 
\subsubsection{The construction} \label{subsec:Construction}
\begin{notation} \label{def: M_p,q} \label{item: M invertible and inverse explicitly} \label{nota:A_p}
We let $R=\R[a_0,\ldots, a_m]$ be a polynomial ring in $m+1$ variables and we let $S=S_{m+1}=\mathrm{Spec}(R)$. We denote by $\sigma_S$  the canonical real structure on $S_{\C}$ and we put \[(\mathbb{P}^1_{S\times_\R B^\star})_\C\coloneqq S_\C\times_\C B_\mathbb{C}^\star\times_\C \mathbb{P}^1_\mathbb{C}=\mathrm{Proj}_{R_\C[t]}(R_\C[t^{\pm 1}][w_0,w_1]).\]
We let $P=\sum_{i=0}^m a_i t^i \in R[t]$, $h=tP^2\in R[t]$ and we let
	\[M= \left(\begin{array}{cc}
		1-h & tPh^n \\
		-Ph^n & \sum_{j=0}^{2n}h^j
	\end{array}\right)\in \mathrm{SL}_2(R[t]) \quad \textrm{and} \quad 
	  A =\left(\begin{array}{cc}
		\sum_{j=0}^{n} h^j & -tP \\
		-P\sum_{j=0}^{n-1}h^j &1
	\end{array}\right) \in \mathrm{SL}_2(R[t]).
	\]
	A direct verification shows that the following identities hold in $\mathrm{Mat}_{2\times 2}(R[t])$:
	\begin{equation}\label{eq:identity-AM} 
		M\left(\begin{smallmatrix} 0 & t \\ 1 & 0 \end{smallmatrix}\right)M=\left(\begin{smallmatrix} 0 & t \\ 1 & 0 \end{smallmatrix}\right) \quad \textrm{and}\quad   
	       AM=\left(\begin{smallmatrix} 0 & t \\ 1 & 0 \end{smallmatrix}\right) A\left(\begin{smallmatrix} 0 & t \\ 1 & 0 \end{smallmatrix}\right).
	\end{equation}  
For every $\R$-rational point $s=(s_0,\ldots, s_m)\in \R^m$ of $S$ (resp. $\C$-rational point $s=(s_0,\ldots, s_m)\in \C^m$ of $S_\C$) we denote by $M(s)$ the matrix in $\mathrm{SL}_2(\R[t])$ (resp. in $\mathrm{SL}_2(\C[t])$) associated to the evaluations $P(s)(t)=\sum_{i=0}^m s_it^i$ and $h(s)(t)=t(P(s)(t))^2$ of the polynomials $P$ and $h=tP^2$ at $s$, viewed as elements of $\R[t]$ (resp. of $\C[t]$). We employ similar notation for the evaluations of the matrix $A$ at $\R$-rational and $\C$-rational points of $S$ and $S_\C$ respectively. 
\end{notation}

\begin{lemma}\label{lem:properties-phi_pq} 
Let $\Psi_M$ and $\Psi_A$ be the $S_\C\times_\C B_\mathbb{C}^\star$-automorphisms of $(\mathbb{P}^1_{S\times_\R B^\star})_\C$ associated as in Notation~\ref{nota:PGL} to the matrices $M$ and $A$ respectively, viewed as elements of $\mathrm{SL}_2(R_\C[t^{\pm 1}]))$. Then the following hold:
	\begin{enumerate}[label=(\alph*), leftmargin=*]
		\item \label{item: properties-phi_pq 1} There exists a unique $S_\C\times_\C B_\mathbb{C}$-automorphism $\Psi$ of $S_\C\times_\C Y_n$ 
	 such that \[\Psi_M=(\mathrm{id}_{S_\C}\times \beta)^{-1} \circ \Psi|_{S_{\C} \times_\C Y_n^\star}\circ (\mathrm{id}_{S_\C}\times \beta).\]
		\item \label{item: properties-phi_pq 3} The composition $\Sigma \coloneqq(\sigma_S\times \sigma)\circ \tau_*\Psi$  is a real structure on $S_\C\times_\C Y_n$.	
		 \item  \label{item: properties-phi_pq 2} The composition 
		 \begin{equation} \label{eq:g} g\coloneqq(\mathrm{id}_{S_\C}\times \beta)\circ \Psi_A\circ (\mathrm{id}_{S_\C}\times \beta^{-1})\colon (S_\C\times_\C Y_n^\star,\sigma_S\times \sigma )\to (S_\C\times_\C Y_n^\star,\Sigma)\end{equation} is a real isomorphism of $(S_\C,\sigma_S)$-schemes.    
	\end{enumerate}
\end{lemma} 

\begin{proof} 	
A direct comparison with the form of the matrices in $G_{R_\C,n}^0$ given in Notation~\ref{nota:subgroups} shows that the image of $M$ in $\mathrm{PGL}_2(R_\C[t^{\pm 1}])$ belongs to $G_{R_\C,n}^0$ and hence, that for every closed point $s$ of $S_\C$, the image of $M(s)$ in $\mathrm{PGL}_2(\C[t^{\pm 1}])$ belongs to $G_{\C,n}^0$. Combined with the identification $\beta^*j^*\Aut_{B_\mathbb{C}}(Y_n)=G_{\C,n}^{0}$ given by Lemma~\ref{lemma: explicit forms of matrices in G_n^0}, this implies that the restriction to every closed fibre $\{s\}\times_\C Y_n \cong Y_n$ of $\mathrm{pr}_{S_\C}\colon S_\C\times_\C Y_n\to S_\C$ of the birational  $S_\C\times_\C B_\C$-automorphism $\Psi\coloneqq(\mathrm{id}_{S_\C}\times \beta) \circ \Psi_M \circ (\mathrm{id}_{S_\C}\times \beta)^{-1}$ of $S_\C\times_\C Y_n$ is regular, hence that $\Psi$ is biregular, which proves Assertion~\ref{item: properties-phi_pq 1}. 
  Since $(S_\C\times_\C Y_n^\star,\sigma_S\times \sigma)\cong ((\mathbb{P}^1_{S\times_\R B^\star})_\C, \sigma_S\times \epsilon)$ (see $\S$~\ref{subsub:cb}) is a real Zariski dense subset of ($S_\C\times_\C Y_n,\sigma_S\times \sigma)$, Assertions ~\ref{item: properties-phi_pq 3} and~\ref{item: properties-phi_pq 2} are respectively equivalent to the facts that $\Sigma_M=(\sigma_S\times \epsilon)\circ \tau_*\Psi_M$ is a real structure on $(\mathbb{P}^1_{S\times_\R B^\star})_\C$ and that $\Psi_A\colon ((\mathbb{P}^1_{S\times_\R B^\star})_\C,\sigma_S\times \epsilon)\to ((\mathbb{P}^1_{S\times_\R B^\star})_\C,\Sigma_M)$ is a real isomorphism. 
Letting $\Xi=\mathrm{id}_S\times \iota_{x,y}$ be the $S_\C\times_\C B_\mathbb{C}^\star$-automorphisms of $(\mathbb{P}^1_{S\times_\R B^\star})_\C$ associated to the matrix $\hat{J}=\left(\begin{smallmatrix} 0 & t \\ 1 & 0 \end{smallmatrix}\right) \in \mathrm{SL}_2(R_\C[t^{\pm 1}])$, we have $\sigma_S \times \epsilon=(\sigma_S\times \sigma_{B^\star}\times \sigma_{\mathbb{P}^1_\R})\circ \tau_*\Xi$. Since the entries of $M$ and $\hat{J}$ are elements of $\R[a_0,\ldots,a_m][t]$, the composition 
\[(\sigma_S \times \epsilon)\circ \tau_*\Psi_M\circ (\sigma_S\times \epsilon)^{-1}\circ \Psi_M=(\sigma_S\times \sigma_{B^\star}\times \sigma_{\mathbb{P}^1_\R})\circ \tau_* (\Xi \circ \Psi_M \circ \Xi^{-1}) \circ (\sigma_S\times \sigma_{B^\star}\times \sigma_{\mathbb{P}^1_\R})^{-1}\circ \Psi_M \] 
 is equal to the $S_\C\times_\C B_\mathbb{C}^\star$-automorphism of $(\mathbb{P}^1_{S\times_\R B^\star})_\C$ associated to the matrix $M\hat{J}M\hat{J}\stackrel{\eqref{eq:identity-AM}}{=}\hat{J}^2=\left(\begin{smallmatrix} t&0\\0&t \end{smallmatrix}\right)$ 
hence is the identity automorphism of $(\mathbb{P}^1_{S\times_\R B^\star})_\C$. By~\eqref{eq:cocycle-relative} in $\S$~\ref{subsec:Galois-cocycle-real-struct}, this implies that $\Sigma_M=(\sigma_S\times \epsilon)\circ \tau_*\Psi_M$ is a real structure. For the same reason, 
\[(\sigma_S\times \epsilon)\circ \tau_*\Psi_A \circ(\sigma_S\times \epsilon)^{-1}= (\sigma_S\times \sigma_{B^\star}\times \sigma_{\mathbb{P}^1_\R})\circ \tau_* (\Xi\circ \Psi_A \circ \Xi^{-1}) \circ (\sigma_S\times \sigma_{B^\star}\times \sigma_{\mathbb{P}^1_\R})^{-1}
\] 
equals the $S_\C\times_\C B_\mathbb{C}^\star$-automorphism  $\Psi_M\circ \Psi_A$ of $(\mathbb{P}^1_{S\times_\R B^\star})_\C$ associated to the matrix $\hat{J} A \hat{J}\stackrel{\eqref{eq:identity-AM}}{=}AM$ and hence that $\Sigma_M\circ \tau_*\Psi_A=(\sigma_S\times \epsilon)\circ \tau_*\Psi_M\circ \tau_*\Psi_A=\Psi_A\circ (\sigma_S\times \epsilon)$.
\end{proof}

The varieties $S_\C$ and $Y_n\cong (X_n)_\C$ being quasi-projective, Proposition~\ref{prop:Effective-descent} and the correspondence of $\S$ ~\ref{subsec:Galois-cocycle-real-struct} imply that the real structure $\Sigma=(\sigma_S\times \sigma)\circ \tau_*\Psi$  on $S_\C\times_\C Y_n$ of Lemma~\ref{lem:properties-phi_pq}\ref{item: properties-phi_pq 3} determines a quasi-projective $S$-form 
\begin{equation}\label{eq:basic-form}
f_n\colon \mathcal{X}_n=(S_\C\times_\C Y_n)/\Sigma \to S 
\end{equation}
of $X_n$. Moreover, the real morphism $\mathrm{pr}_{S}\times \pi_\C\colon (S\times_\C Y_n, \Sigma)\to (S_\C\times_\C B_\C,\sigma_S\times \sigma_B)$ corresponds to a morphism of $S$-schemes $\Pi\colon \mathcal{X}_n\to S\times_\R B$  and Lemma~\ref{lem:properties-phi_pq}\ref{item: properties-phi_pq 2} asserts that the real morphism $g$ induces an $S$-isomorphism 
\begin{equation} 
\alpha\colon S\times_\R (X_n\setminus\pi^{-1}(0)) \to \Pi^{-1}(S\times_\R B^\star),
\end{equation}
\emph{i.e.}, that the restriction of $f_n$ to $\Pi^{-1}(S\times_\R B^\star)\to S$ is isomorphic to the trivial $S$-form of $X_n\setminus\pi^{-1}(0)$.

 \bigskip

 The construction extends readily to give an $S$-form  $\tilde{f}_{n}\colon\tilde{\mathcal{X}}_{n}\to S$ of the minimal desingularisation $d\colon \tilde{X}_n\to X_n$ of $X_n$. Namely, letting $d_\C\colon (\tilde{Y}_n,\tilde{\sigma})\to (Y_n,\sigma)$ be the minimal real desingularisation described in $\S$~\ref{subsub:desing}, Proposition~\ref{prop: Aut(X_n)} combined with the same argument as in the proof of Lemma~\ref{lem:properties-phi_pq}\ref{item: properties-phi_pq 1} implies the 
existence of a unique $S_\C$-automorphism $\tilde{\Psi}$ of $S_\C\times_\C \tilde{Y}_n$ such that  $(\mathrm{id}_{S_\C} \times d_\C)\circ\tilde{\Psi}=\Psi\circ (\mathrm{id}_{S_\C}\times d_\C)$. 
Since \[\mathrm{id}_{S_\C}\times_\C d_\C\colon  (S_\C \times_\C \tilde{Y}_n, \sigma_S\times \tilde{\sigma}) \to (S_\C\times_\C Y_n, \sigma_S\times \sigma)\] is a real birational morphism restricting to a real isomorphism $S_\C \times_\C \tilde{Y}_n^\star \cong S_\C \times_\C Y_n^\star$, the automorphism $\tilde{\Psi}$ satisfies the cocycle condition $(\sigma_{S}\times\tilde{\sigma})\circ\tau_{*}\tilde{\Psi}\circ(\sigma_{S}\times\tilde{\sigma})^{-1}\circ\tilde{\Psi}=\mathrm{id}_{S_\C \times_\C \tilde{Y}_n}$.  It thus determines a real structure $\tilde{\Sigma}=(\sigma_S\times \tilde{\sigma})\circ \tau_*\tilde{\Psi}$ on $S_\C\times_\C \tilde{Y}_n$ for which  the morphisms
 $\mathrm{pr_S}\colon (S_\C\times_\C \tilde{Y}_n,\tilde{\Sigma})\to (S_\C,\sigma_S)$, $(\mathrm{id}_S\times d_\C)\colon (S_\C\times_\C \tilde{Y}_n,\tilde{\Sigma})\to (S_\C\times_\C Y_n,\Sigma)$ and $\mathrm{pr}_{S}\times_\C \tilde{\pi}_\C\colon (S\times_\C \tilde{Y}_n, \tilde{\Sigma})\to (S_\C\times_\C B_\C,\sigma_S\times \sigma_B)$ are real and such that 
 \begin{equation}\label{eq:gtilde} \tilde{g}=(\mathrm{id}_{S_\C}\times d_\C^{-1}) \circ g \circ  (\mathrm{id}_{S_\C}\times d_\C) \colon (S_\C\times_\C \tilde{Y}_n^\star,\sigma_S\times \tilde{\sigma} )\to (S_\C\times_\C \tilde{Y}_n^\star,\tilde{\Sigma})\end{equation} is a real isomorphism of $(S,\sigma_S)$-schemes. 
 By the same argument as above, these data correspond to a quasi-projective $S$-form $\tilde{f}_{n}\colon\tilde{\mathcal{X}}_{n}\to S$, 
a birational $S$-morphism $D\colon \tilde{\mathcal{X}}_{n}\to \mathcal{X}_n$ induced by  $\mathrm{id}_{S_\C}\times d_\C$ which is a desingularisation of $\mathcal{X}_n$ and morphism of $S$-schemes $\tilde{\Pi}=\Pi \circ D\colon \tilde{\mathcal{X}}_n\to S\times B$ for which 
\begin{equation} \label{eq:rel-diffbir}
\tilde{\alpha}= D^{-1} \circ \alpha \circ (\mathrm{id}_S\times d)\colon  S\times_\R (\tilde{X}_n\setminus\pi^{-1}(0))\to \tilde{\Pi}^{-1}(S\times_\R B^\star)
\end{equation}
is an isomorphism of $S$-schemes.

\subsubsection{Isomorphism classes of real fibres}
We keep the notation of the previous subsection and now determine the isomorphism classes of the fibres over $\R$-rational points of $S$ of the $S$-form $f_n\colon \mathcal{X}_n\to S$ of $X_n$ and of the $S$-form $\tilde{f}_n\colon \tilde{\mathcal{X}}_n\to S$ of its minimal desingularisation $\tilde{X}_n$. By construction of the real structure $\Sigma=(\sigma_S\times \sigma)\circ \tau_*\Psi$, for every $\R$-rational point $s=(s_0,\ldots, s_m)\in S(\R)=\R^{m+1}$, the scheme-theoretic fibre $\mathcal{X}_{n,s}= f_n^{-1}(s)$ of $f_n$ equals the real form of $X_n$ corresponding to the real structure $\sigma\circ \tau_*\Psi_s$ associated to the $\C$-automorphism $\Psi_s$ of $\{s\}\times_\C Y_n\cong Y_n$ determined by the matrix $M(s)\in \mathrm{SL}_2(\R[t^{\pm 1}])$ (viewed as an element of $\mathrm{SL}_2(\C[t^{\pm 1}]$) associated to the evaluation $P(s)(t)=\sum_{i=0}^m s_it^i \in \R[t]$ at $s$ of the polynomial $P=\sum_{i=0}^m a_it^i \in R[t]$. The same description holds for the fibres of the  $S$-form $\tilde{f}_n\colon \tilde{\mathcal{X}}_n\to S$ of $\tilde{X}_n$ over $\R$-rational point of $S$, with $\Sigma$, $\Psi$ and $\sigma$ replaced by  $\tilde{\Sigma}$, $\tilde{\Psi}$ and $\tilde{\sigma}$ respectively. 
 
\begin{proposition} \label{prop:isoclasses} For a pair of $\R$-rational points $s,s'\in S(\R)$, the following are equivalent: 
\begin{enumerate}[label=(\alph*), leftmargin=*]
 \item The real surfaces $\mathcal{X}_{n,s}$ and $\mathcal{X}_{n,s'}$ are isomorphic. \label{prop:isoclasses, item 1}
 \item The real surfaces $\tilde{\mathcal{X}}_{n,s}$ and $\tilde{\mathcal{X}}_{n,s'}$ are isomorphic. \label{prop:isoclasses, item 2}
 \item There exists $e\in \R^*$ such that $ P(s')(t)\equiv e P(s)(e^2t)$ modulo $t^n$. \label{prop:isoclasses, item 3}
 \end{enumerate}
\end{proposition} 
\begin{proof} 
The surfaces $\mathcal{X}_{n,s}$ and $\tilde{\mathcal{X}}_{n,s}$ correspond under the equivalence of categories of Proposition~\ref{lem:Descent-on-functors} to the real surfaces $(Y_n,\sigma\circ \tau_*\Psi_s)$ and $(\tilde{Y}_n,\tilde{\sigma}\circ \tau_*\tilde{\Psi}_s)$, and similarly for $s'$. The equivalence of~\ref{prop:isoclasses, item 1} and~\ref{prop:isoclasses, item 2} follows from the definition of $\tilde{\sigma}$ and $\tilde{\Psi}$ and the fact that, by
Proposition~\ref{prop: Aut(X_n)}, the minimal desingularisation morphism $d_\C\colon\tilde{Y}_n\to Y_n$ induces an isomorphism $d_\C^*\colon\Aut_{\mathbb{C}}(Y_n) \to \Aut_{\mathbb{C}}(\tilde{Y}_n)$.

We now prove the equivalence of~\ref{prop:isoclasses, item 1} and~\ref{prop:isoclasses, item 3}. By definition, the $\C$-automorphism $\Psi_{s}$ of $Y_n$ corresponds under the isomorphism $\Aut_\mathbb{C}(Y_n)\cong G_{\C,n}\rtimes (\C^*,\times)$ of Proposition~\ref{prop: Aut(X_n)} to the class of the pair $(M(s),1)$.  Recall by $\S$~\ref{subsec:Galois-cocycle-real-struct} that viewing $\mathrm{PGL}_2(\mathbb{C}[t^{\pm 1}])\rtimes (\C^*,\times) $  and its subgroup $G_{\C,n}\rtimes (\C^*,\times) $ as endowed with the $\Gamma$-group structure defined by the $\Gamma$-action described in~\eqref{eq:Gamma-group-struct} in the proof of Proposition~\ref{cor:Real-Aut}, the real surfaces $(Y_n,\sigma\circ \tau_*\Psi_s)$ and $(Y_n,\sigma\circ \tau_*\Psi_{s'})$ are isomorphic if and only if the Galois $1$-cocycles  $c_s,c_{s'}\colon\Gamma=\{\pm 1\} \longrightarrow  G_{\C,n}\rtimes (\C^*,\times)$ defined by mapping $-1$ to the classes of $(M(s),1)$ and $(M(s'),1)$ in $G_{\C,n}\rtimes (\C^*,\times)$ respectively have the same class in the pointed Galois cohomology set $H^1(\Gamma,G_{\C,n}\rtimes (\C^*,\times))$. Denote by $Q=(\PGL_2(\C[t^{\pm 1}])\rtimes (\C^*,\times))/(G_{\C,n} \rtimes (\C^*,\times))$ the quotient set and consider the following part of the associated Galois cohomology sequence of pointed sets 
\[(\PGL_2(\C[t^{\pm 1}])\rtimes (\C^*,\times))^{\Gamma} \to Q^{\Gamma} \overset{\partial}{\rightarrow} H^1(\Gamma,G_{\C,n}\rtimes (\C^*,\times)).\]
By \cite{BS64}, Proposition 1.12], the connecting map $\partial$ sends a class $[c] \in Q^{\Gamma}$ represented by an element $c\in G_{\C,n}\rtimes (\C^*,\times)$ to the class in $ H^1(\Gamma, \PGL_2(\C[t^{\pm 1}])\rtimes (\C^*,\times))$ of the $1$-cocycle $c^{-1} \circ \sigma^{-1} \circ c \circ \sigma$. Moreover, $\partial([c])=\partial([c'])$ if and only if there exists $b \in (\PGL_2(\C[t^{\pm 1}])\rtimes (\C^*,\times))^\Gamma$ such that $c=bc'$.

Letting $A(s)$ be the matrix of Notation~\ref{nota:A_p}, the equality \[(\left(\begin{smallmatrix} 0 & t \\ 1 & 0 \end{smallmatrix}\right) A(s)\left(\begin{smallmatrix} 0 & t \\ 1 & 0 \end{smallmatrix}\right),1)=(A(s),1)\cdot (M(s),1)\]
 in $\mathrm{GL}_2(\mathbb{C}[t^{\pm 1}]\rtimes (\C^*,\times)$ deduced from~\eqref{eq:identity-AM} implies that the class of $(A(s),1)$ in $Q$ belongs to $Q^\Gamma$ and that the cohomology class $[c_s]$ of $c_s$ is the image by $\partial$ of the class of $(A(s),1)$. The same description holds for $c_s'$ and we get that  $[c_s]=[c_{s'}]$ if and only if there exists an element $(B,\nu)\in \mathrm{GL}_2(\C[t^{\pm 1}])\rtimes (\C^*,\times)$ representing a class in $(\PGL_2(\C[t^{\pm 1}])\rtimes (\C^*,\times))^\Gamma$ with the property that the pair  
\[(D,\nu)\coloneqq (A(s),1)^{-1}\cdot (B,\nu)\cdot (A(s'),1) \in \mathrm{GL}_2(\C[t^{\pm 1}])\rtimes (\C^*,\times)\] represents an element of $G_{\C,n} \rtimes (\C^*,\times)$. By the proof of Proposition~\ref{cor:Real-Aut}, an element of $(\PGL_2(\C[t^{\pm 1}])\rtimes (\C^*,\times))^{\Gamma}$ is either the class of a pair $(\left(\begin{smallmatrix}
	\lambda & 0 \\ 0 & 1
\end{smallmatrix}\right), \lambda\overline{\lambda})$ or the class of a pair $(\left(\begin{smallmatrix}
0 & \lambda t \\ 1 & 0
\end{smallmatrix}\right), \lambda\overline{\lambda})$, where $\lambda \in \C^*$. 

We now consider the case where $(B,\nu)=(\left(\begin{smallmatrix} \lambda & 0 \\ 0 & 1 \end{smallmatrix}\right), \lambda\overline{\lambda})$, the other being similar and left to the reader. Since $B$, $A(s)$ and $A(s')$ belong to $\mathrm{GL}_2(\mathbb{C}[t])$, so does the matrix $D$, and a direct calculation renders that $D(0)=
 \left(\begin{smallmatrix} \lambda & 0 \\ \lambda q_1(0)-q_2(0) & 1 \end{smallmatrix}\right)$. This implies that the class  of $(D,\nu)$ belongs to $G_{\C,n}\rtimes (\C^*,\times)$ if and only if it lies in $G_{\C,n}^0\rtimes (\C^*,\times)$. By Notation~\ref{nota:subgroups}, this holds if and only if the upper right entry is equal to $t^{n+1}\beta$ and the lower left entry is equal to $t^n\gamma$ for some $\beta, \gamma \in \C[t]$. By direct calculation, the upper right and lower left entries of $D$ are respectively equal to 
 \[ -\lambda t(P(s')(t)-\overline{\lambda}P(s)(\lambda\overline{\lambda}t)) \quad \textrm{and} \quad (\sum_{i=0}^{n-1}h(s)(\lambda\overline{\lambda}t)^i)(\sum_{j=0}^{n-1}h(s')(t)^j)(P(s')(t)-\lambda P(s)(\lambda \overline{\lambda}t))+t^n\chi\]
 for some $\chi \in \C[t]$. Since $P(s),P(s')\in \mathbb{R}[t]$, the polynomial $P(s')(t)-\lambda P(s)(\lambda \overline{\lambda}t)$ is the conjugate of $P(s')(t)-\overline{\lambda}P(s)(\lambda\overline{\lambda}t)$ and we conclude that the class of $(D,\nu)$ belongs to $G_{\C,n}\rtimes (\C^*,\times)$ if and only if  $\lambda=e \in \R^*$ and $P(s')(t) - e P(s)(e^2t)\in t^n\R[t]$.
\end{proof}

\subsubsection{Birational homeomorphisms and diffeomorphisms of real loci}\label{subsec:diffbir}

Recall that a \emph{birational diffeomorphism} between two smooth real algebraic varieties $X$ and $X'$  with non-empty loci is a birational map $f\colon X\dashrightarrow X'$ such that $f$ and its birational inverse $f^{-1}$ are defined at every $\R$-rational point of $X$ and $X'$ respectively. Such a birational map induces a diffeomorphism  $f(\R)\colon X(\R)\to  X'(\R)$ between $X(\R)$ and $X'(\R)$ endowed with their respective structure of differentiable manifolds locally inherited from the standard differentiable manifold structure on $\mathbb{A}^n_{\mathbb{R}}(\R)=\mathbb{R}^n$. A birational map $f\colon X'\dashrightarrow X$ between normal real algebraic varieties is called a \emph{birational homeomorphism} if the continuous map $f|_{\mathrm{dom}(f)}(\R)\colon \mathrm{dom}(f)(\R)\to X(\R)$ induced by the restriction of $f$ to its domain of definition $\mathrm{dom}(f)$ extends to a homeomorphism $f(\R)\colon X'(\R)\to X(\R)$ with inverse induced by the birational inverse $f^{-1}$ of $f$. 

\begin{proposition}\label{prop:diffbir} All the $S$-maps in the commutative diagram 
\[\xymatrix {S \times \tilde{X}_n \ar[d]_{\mathrm{id}_S\times d} \ar@{-->}[r]^{\tilde{\alpha}} & \tilde{\mathcal{X}}_n \ar[d]^{D} \\ S\times X_n \ar@{-->}[r]^{\alpha} & \mathcal{X}_n}\]  
are birational homeomorphisms. 
\end{proposition}
\begin{proof}
By the description given in $\S$~\ref{subsub:desing}, the minimal real desingularisation $d_\C\colon (\tilde{Y}_n,\tilde{\sigma})\to (Y_n,\sigma)$ induces a bijection between the sets of real points of $(\tilde{Y}_n,\tilde{\sigma})$ and $(Y_n,\sigma)$. This implies that $d(\R)\colon \tilde{X}_n(\R)\to X_n(\R)$ is a proper bijective continuous map between $\tilde{X}_n(\R)$ and $X_n(\R)$ endowed with their respective Euclidean topologies, hence is a homeomorphism. It follows in turn that $\mathrm{id}_S\times d$ is a birational homeomorphism inducing a homeomorphism between $S(\R)\times \tilde{X}_n(\R)$ and $S(\R)\times X_n(\R)$ endowed with their respective Euclidean topologies. 

By the construction of $\S$~\ref{subsec:Construction}, for every $\R$-rational point $s=(s_0,\ldots, s_m)\in S(\R)$, the morphism $D_s\colon \tilde{\mathcal{X}}_{n,s}\to \mathcal{X}_{n,s}$ corresponds under the equivalence of categories of Proposition~\ref{lem:Descent-on-functors} to the real morphism \[d_\C\colon (\tilde{Y}_n,\tilde{\sigma}\circ \tau_*\tilde{\Psi}_s)\to (Y_n,\sigma\circ \tau_*\Psi_{s}),\] where $\Psi_s$ is the $B_\C$-automorphism $\Psi_{M(s)}$ of $Y_n$ associated to the evaluation $M(s)$ of the matrix $M$ of Notation~\ref{def: M_p,q} at $s$ and where $\tilde{\Psi}_s$ is the image of $\Psi_s$ by the isomorphism $d_\C^*\colon \mathrm{Aut}_\C(Y_n)\to \mathrm{Aut}_\C(\tilde{Y}_n)$ of Proposition~\ref{prop: Aut(X_n)}. Since the class of $M(s)$ in $\mathrm{PGL}_2(\C[t^{\pm 1}])$ belongs to the subgroup $G_{\C,n}^0$, it follows that $\Psi_s \in \mathrm{Aut}_{B_\C}^0(Y_n)$ and hence that $(0,[0:0:1])$ is the unique real point of $(Y_n,\sigma\circ \tau_*\Psi_s)$ supported on $\pi_\C^{-1}(0)$. If $n=0$ then $d_\C$ is an isomorphism. If $n\geq 1$, then $Y_n$ is singular and, by the description given in $\S$~\ref{subsub:desing}, the total transform $\tilde{\pi}_\C^{-1}(0)$ of $\pi_\C^{-1}(0)$ in $\tilde{Y}_n$ is a chain of rational curves having the proper transforms $L_x$ and $L_y$ as boundary curves with self-intersection $(-1)$ exchanged by $\tilde{\sigma}$ and $n$ pairs of curves $E_{i,x}$ and $E_{i,y}$, $i=0,\ldots, n-1$ exchanged by $\tilde{\sigma}$. 
The unique real point $r=E_{0,x}\cap E_{0,y}$ of $(\tilde{Y}_n,\tilde{\sigma})$ supported on $\tilde{\pi}_\C^{-1}(0)$ is also the unique real points of $(\tilde{Y}_n,\tilde{\sigma}\circ \tau_*\tilde{\Psi}_s)$ supported on this fibre. Since $d_\C$ restricts to an isomorphism over $B_\C^\star$, we infer as in the previous situation that $D_s(\R)\colon \tilde{\mathcal{X}}_{n,s}(\R)\to \mathcal{X}_{n,s}(\R)$ is bijective for every $s$ and hence, since $D$ is a proper morphism, that $D(\R)\colon \tilde{\mathcal{X}}_{n,s}(\R)\to \mathcal{X}_{n,s}(\R)$ is a homeomorphism between $\tilde{\mathcal{X}}_{n,s}(\R)$  and $\mathcal{X}_{n,s}(\R)$ endowed with their respective Euclidean topologies.

By definition, $\alpha$ and $\tilde{\alpha}$ are the birational $S$-maps corresponding respectively to the real morphism $g$ in~\eqref{eq:g} and $\tilde{g}$ in~\eqref{eq:gtilde} viewed as real birational $(S_\C,\sigma_S)$-maps. For every $\R$-rational point $s=(s_0,\ldots, s_m) \in S(\R)$, the image of the class in $\mathrm{PGL}_2(\C[t^{\pm 1}])$ of the evaluation $A(s)$  of $A$ at $s$ by the homomorphism $c^*\colon\mathrm{PGL}_2(\mathbb{C}[t^{\pm 1}]) \hookrightarrow \mathrm{PGL}_3(\mathbb{C}[t^{\pm 1}])$ (see~\eqref{eq:PGL2-embeds-PGL3} in the proof of Lemma~\ref{lemma: explicit forms of matrices in G_n^0}) is the class in $\mathrm{PGL}_3(\mathbb{C}[t^{\pm 1}])$ of the element
	\begin{align*}
		N(s)=\left(\begin{array}{ccc}
			(\sum_{j=0}^{n} h(s)^j)^2 & (-tP(s) )^2t^{-1} & 2(\sum_{j=0}^{n} h(s)^j) (-tP(s) )t^n \\
			(-p\sum_{j=0}^{n-1}h(s)^j)^2t & 1 & 2(-P(s)\sum_{j=0}^{n-1}h(s)^j)t^{n+1} \\
			(\sum_{j=0}^{n} h(s)^j)(-P(s)\sum_{j=0}^{n-1}h(s)^j)t^{-n} & (-tP(s))t^{-(n+1)} & (\sum_{j=0}^{n} h(s)^j)+(-tP(s))(-P(s)\sum_{j=0}^{n-1}h(s)^j)
		\end{array}\right)
	\end{align*}
of $\mathrm{GL}_3(\C[t^{\pm 1}])$. 
If $P(s)\in t^n\mathbb{R}[t]$ then $N(s)\in \mathrm{Mat}_{3,3}(\mathbb{C}[t])\setminus \mathrm{Mat}_{3,3}(t\mathbb{C}[t])$ and, by the definition of $\Psi_{A(s)}$, we obtain that $g_s|_{t=0}$ is the restriction to $\pi_\C^{-1}(0)$ of an automorphism of $\mathbb{P}^2_{\mathbb{C}}$ of the form $[x:y:z]\mapsto [x:y:z+\ell(x,y)]$ for some linear form $\ell(x,y)\in \mathbb{C}[x,y]$. This implies that $g_s$ is an automorphism of $Y_n$ and Proposition~\ref{prop: Aut(X_n)} implies in turn that $\tilde{g}_s$ is an automorphism of $\tilde{Y}_n$. Thus, $\alpha_s\colon \{s\}\times_\C X_n\to \mathcal{X}_{n,s}$ and $\tilde{\alpha}_s\colon \{s\}\times_\C \tilde{X}_n\to \tilde{\mathcal{X}}_{n,s}$ are both isomorphisms. Now assume that $n\geq 1$ and that  $P(s)\in \mathbb{R}[t]\setminus t^n\mathbb{R}[t]$. Then, by multiplying $N(s)$ by the appropriate power $t^m$, $m>1$, so that it becomes an element of $\mathrm{Mat}_{3,3}(\mathbb{C}[t])\setminus \mathrm{Mat}_{3,3}(t\mathbb{C}[t])$, and then evaluating at $t=0$, we obtain that $g_s$ is a strictly birational map which has $(0,[0:0:1])$ as unique proper base point and which contracts $\pi_\C^{-1}(0)=L_x\cup L_y$ onto the point $(0,[0:0:1])$. This implies in turn that $\tilde{g}_s$ contracts the proper transform of $\pi_\C^{-1}(0)$ in $\tilde{Y}_n$ whence is not an isomorphism. The very same argument shows that the birational inverse $\tilde{g}_s^{-1}$ of $\tilde{g}_s$ also contracts the proper transform of $\pi_\C^{-1}(0)$ in $\tilde{Y}_n$. We now argue that $\tilde{g}_s$ is a birational diffeomorphism. Assume  that the unique real point $r$ of $(\tilde{Y}_n,\tilde{\sigma})$ supported on $\tilde{\pi}_\C^{-1}(0)$ is a proper base point of the real birational map $\tilde{g}_s\colon (\tilde{Y}_n,\tilde{\sigma})\dashrightarrow (\tilde{Y}_n,\tilde{\sigma}\circ \tau_*\tilde{\Psi}_s)$ and consider its minimal real resolution of indeterminacy 
	\[(\tilde{Y}_n,\tilde{\sigma})\stackrel{\gamma}{\leftarrow} (Z,\tilde{\sigma}') \stackrel{\gamma'}{\rightarrow} (\tilde{Y}_n,\tilde{\sigma}\circ \tau_*\tilde{\Psi}_s).\]
	Since $r$ is a smooth base point of $\gamma^{-1}$, $\gamma^{-1}$ factors through the blow-up of $r$, say with real exceptional divisor  $(E,\delta)$ isomorphic to $(\mathbb{P}^1_\mathbb{C},\sigma_{\mathbb{P}^1_\mathbb{R}})$. This implies in particular that the proper transforms of $E_{0,x}$ and $E_{0,y}$ in $Z$ have self-intersection number $\leq -3$ and that they either both intersect $E$ or intersect a pair of non-real $\gamma$-exceptional curves exchanged by $\tilde{\sigma}'$. 
	On the other hand, since $\tilde{g}_s$ contracts the proper transform of $\pi_\C^{-1}(0)$, $\gamma'$ is not an isomorphism. The minimality assumption implies further  that the proper transform of $E$ in $Z$ has self-intersections $\leq -2$ and that $\gamma'$ consists of the contraction of the proper transform of $\pi_\C^{-1}(0)$, followed by a sequence of contractions of either pairs of disjoint $(-1)$-curves exchanged by the successively induced real structures or irreducible $(-1)$-curves which are stable under these induced real structures. But then, after the contraction $(Z,\tilde{\sigma}')\to (Z',\tilde{\sigma}'')$ of all such possible successive real curves which are not $\gamma$-exceptional, the images of $E_{0,x}$ and $E_{0,y}$ and $E$ are curves with self-intersection $\leq -2$ contained in $\gamma^{-1}(\tilde{\pi}_\C^{-1}(0))$. It follows that the induced real birational morphism $(Z',\tilde{\sigma}'')\to (\tilde{Y}_n,\tilde{\sigma}\circ \tau_*\tilde{\Psi}_s)$ does not contract $E$, hence that $\tilde{\pi}_\C^{-1}(0)$ endowed with the restriction of the real structure $\tilde{\sigma}\circ \tau_*\tilde{\Psi}_s$ contains a real curve isomorphic to $(\mathbb{P}^1_\mathbb{C},\sigma_{\mathbb{P}^1_\mathbb{R}})$. But this is impossible since, as observed above, $r$ is the unique real point of $(\tilde{Y}_n;\tilde{\sigma}\circ \tau_*\tilde{\Psi}_s)$ supported on $\tilde{\pi}_\C^{-1}(0)$. Thus, $\tilde{g}_s$ is defined at $r$ and the very same argument implies that its birational inverse is also defined at $r$. Since $g_s$ is on the other hand a real isomorphism off $\tilde{\pi}_\C^{-1}(0)$, it follows that $\tilde{\alpha}_s\colon \{s\}\times_\C \tilde{X}_n\dashrightarrow \tilde{\mathcal{X}}_{n,s}$ is a birational diffeomorphism. This implies in turn that $\tilde{\alpha}$ and its birational inverse are defined at every point of $(S\times_\C \tilde{X}_n)(\R)$ and $\tilde{\mathcal{X}}_n(\R)$ respectively, showing that $\alpha$ is a birational diffeomorphism. 
 
   To complete the proof it now suffices to observe that the birational $S$-map $\alpha=D \circ \tilde{\alpha} \circ (\mathrm{id}_S\times d)^{-1}$ is a composition of birational homeomorphism, hence is a birational homeomorphism. 	 
\end{proof}

\subsubsection{Proofs of Theorem~\ref{MainThm} and Theorem~\ref{MainThm-2}} 

We now proceed to the proofs of Theorem~\ref{MainThm} and Theorem~\ref{MainThm-2}. For every $n\geq 2$, we consider the $S$-forms $f_n\colon \mathcal{X}_n\to S$ and $\tilde{f}_n\colon \tilde{\mathcal{X}}_n\to S$ constructed in $\S$~\ref{subsec:Construction} over the particular choice of $\R$-scheme $S=S_{n}=\mathrm{Spec}(\R[a_0,\ldots,a_{n-1}])\cong \mathbb{A}^n_\R$. We let $T_{n-1}\cong \mathbb{A}^{n-1}_\R$ be the closed subscheme of $S_{n}$ with defining ideal $(a_0-1)\mathbb{R}[a_0,\ldots, a_{n-1}]$ and we let $h_n\colon \mathcal{Z}_n\to T_{n-1}$ be the $T_{n-1}$-form of either $X_n$ or $\tilde{X}_n$ obtained by restricting either $f_n\colon \mathcal{X}_n\to S_n$ or $\tilde{f}_n\colon \tilde{\mathcal{X}}_n\to S_n$ over $T_{n-1}$. The following proposition implies Theorem~\ref{MainThm}:

\begin{proposition}\label{prop:main-prop} 
The morphism $h_n\colon \mathcal{Z}_n\to T_{n-1}$ is a $T_{n-1}$-form of $X_n$ (resp. of $\tilde{X}_n$) birationally homeomorphic over $T_{n-1}$ to the trivial $T_{n-1}$-form $T_{n-1}\times_\R X_n$ (resp. birationally diffeomorphic over $T_{n-1}$ to $T_{n-1}\times_\R \tilde{X}_n$) and with pairwise non-isomorphic real fibres. 
\end{proposition}

\begin{proof} 
The first part of the assertion follows from Proposition~\ref{prop:diffbir} which implies that the restriction of the birational $S$-homeomorphism $\alpha$ (resp. of the birational $S$-diffeomorphism $\tilde{\alpha}$) over $T_{n-1}$ is a birational homeomorphism $T_{n-1}\times X_n\dashrightarrow \mathcal{Z}_n$ (resp. a birational diffeomorphism  $T_{n-1}\times \tilde{X}_n\dashrightarrow \mathcal{Z}_n$). 
For the second part, we observe that $T_{n-1}(\R)$ equals the subset of $S_n(\R))\cong \R^n$ consisting of points $s=(1,s_1,\ldots,s_{n-1})$ and that the fibre $\mathcal{Z}_{n,s}$ of $h_n$ over $s$ equals the real form $\mathcal{X}_{n,s}$ of $X_n$ (resp. of $\tilde{\mathcal{X}}_{n,s}$ of $\tilde{X}_n$) corresponding to the evaluation $P(s)$ of $P=\sum_{i=0}^{n-1}a_it^i$ at $s$. Since $P(s)(0)=P(s')(0)=1$ for every $s,s'\in T_{n-1}(\R)$, Proposition ~\ref{prop:isoclasses} implies that the fibres $\mathcal{Z}_{n,s}$ and $\mathcal{Z}_{n,s'}$ of $h_n$ over $s$ and $s'$ are isomorphic if and only if $s=s'$.
\end{proof}

 Now let $d\geq 3$ and let $U\subset \mathbb{P}^{d-2}_\mathbb{R}$ be the complement of the union of a hyperplane $H\cong \mathbb{P}^{d-3}_\R$ and a smooth hypersurface $V$ of even degree $n\geq d-1$ intersecting $H$ transversely and such that $V(\R)=\emptyset$. Then $V_\C\cup H_\C$ is an SNC divisor of degree $\geq d$ on $\mathbb{P}^{d-2}_\C$ which implies that $U_\C$ is a smooth affine variety of log-general type. Furthermore, since $V(\R)=\emptyset$, the inclusion $U\hookrightarrow   \mathbb{P}^{d-2}_\mathbb{R}\setminus H\cong \mathbb{A}^{d-2}_\R$ is a birational diffeomorphism. Given any integer $m\geq 1$, let $h_{m+1}\colon \mathcal{Z}_{m+1}\to T\coloneqq T_{m}$ be the $T$-form of $\tilde{X}_{m+1}$ constructed above and let \[h=h_{m+1}\circ \mathrm{pr}_1\colon \mathcal{Z}\coloneqq \mathcal{Z}_{m+1}\times U\rightarrow T.\] Theorem~\ref{MainThm-2} is then a consequence of the following proposition:

\begin{proposition} \label{prop:Main-Cor}
The morphism $h\colon \mathcal{Z}\to T=\mathbb{A}^m_\R$ is a $T$-form of $\tilde{X}_{m+1}\times_\R U$ birationally diffeomorphic over $T$ to the trivial $T$-form $T\times_\R (\tilde{X}_{m+1}\times_\R U)$ and whose real fibres are pairwise non-isomorphic.
\end{proposition}
\begin{proof} The first part of the assertion follows from Proposition~\ref{prop:diffbir} and the construction of $U$. 
Let $s,s' \in T(\R)=\mathbb{A}^m_\R(\R)$ be $\R$-rational points of $T$ and let $\Theta\colon \mathcal{Z}_{s}\to  \mathcal{Z}_{s'}$ be an isomorphism. By construction, \[\Theta_\C\colon (\tilde{Y}_{m+1}\times_\C U_\C,(\tilde{\sigma}\circ \tau_*\tilde{\Psi}_{s})\times \sigma_U)\rightarrow  (\tilde{Y}_{m+1}\times_\C U_\C,(\tilde{\sigma}\circ \tau_*\tilde{\Psi}_{s'})\times \sigma_U),\]
where $\sigma_U$ denotes the canonical real structure on $U_\C$, is a real isomorphism. Since $U_\C$ is of log-general type whereas $\tilde{Y}_{m+1}$ is $\mathbb{P}^1$-ruled (see $\S$~\ref{subsub:cb}), it follows from Iitaka-Fujita strong cancellation theorem~\cite{IiFu77} that there exists a unique automorphism $\theta$ of $U_\C$ such that $\mathrm{pr}_{U_\C}\circ \Theta_\C=\theta\circ \mathrm{pr}_{U_\C}$. Since $\Theta_\C$ is a real isomorphism, the uniqueness of $\theta$ implies that $\theta\colon (U_\C,\sigma_U)\to (U_\C,\sigma_U)$ is real as well. Thus, for any real point $u\in U_\C(\C)^{\sigma_U}\cong \R^{d-2}$ of $(U_\C,\sigma_U)$, $\Theta_\C$ induces a real isomorphism between the fibre $(\tilde{Y}_{m+1},\tilde{\sigma}\circ \tau_*\tilde{\Psi}_{s})$ of $
\mathrm{pr}_{U_\C}$ over $u$ and the fibre $(\tilde{Y}_{m+1},\tilde{\sigma}\circ \tau_*\tilde{\Psi}_{s'})$ of $\mathrm{pr}_{U_\C}$ over $\theta(u)\in U_\C(\C)^{\sigma_U}$ which implies, by Proposition~\ref{prop:main-prop}, that $s=s'$. 
\end{proof}

\end{document}